\documentclass{amsproc}
\usepackage{euscript}
\usepackage{cases}
\usepackage{mathrsfs}
\usepackage{bbm}
\usepackage{amssymb}
\usepackage{txfonts}
\usepackage{amsfonts,amsmath,amsxtra,mathdots,mathabx}
\usepackage{color}
\usepackage{hyperref}
\usepackage{tikz}

\usepackage{floatrow}
\newfloatcommand{capbtabbox}{table}[][\FBwidth]

\allowdisplaybreaks

\DeclareFontFamily{U}{matha}{\hyphenchar\font45}
\DeclareFontShape{U}{matha}{m}{n}{
	<5> <6> <7> <8> <9> <10> gen * matha
	<10.95> matha10 <12> <14.4> <17.28> <20.74> <24.88> matha12
}{}
\DeclareSymbolFont{matha}{U}{matha}{m}{n}

\DeclareMathSymbol{\Lt}{3}{matha}{"CE}
\DeclareMathSymbol{\Gt}{3}{matha}{"CF}

\DeclareSymbolFont{mathc}{OML}{txmi}{m}{it}
\DeclareMathSymbol{\varvv}{\mathord}{mathc}{118}
\DeclareMathSymbol{\varnu}{\mathord}{mathc}{"17}
\DeclareMathSymbol{\varww}{\mathord}{mathc}{119}
\DeclareSymbolFont{mathd}{OML}{ztmcm}{m}{it}
\DeclareMathSymbol{\varalpha}{\mathord}{mathd}{11}
\DeclareMathSymbol{\varlambda}{\mathord}{mathd}{21}
\def\valpha{\text{\scalebox{0.85}{$\varalpha$}}}

\DeclareMathSymbol{\depsilon}{\mathord}{mathd}{15}
\def\vepsilon{\text{\scalebox{0.88}{$\depsilon$}}}

\DeclareMathSymbol{\varchi}{\mathord}{mathd}{31}

 \newcommand{\BR}{{\mathbb {R}}}

\def\CalJ{\text{\usefont{U}{dutchcal}{m}{n}J}\hskip 0.5pt}
\def\CalL{\text{\usefont{U}{dutchcal}{m}{n}L}\hskip 0.5pt}

\newcommand{\ds}{\displaystyle}

\newcommand{\ra}{\rightarrow}

\def\-{^{-1}}

\def\mod{\mathrm{mod}\, }

\def\vv{\varv}
\def\vw{\varw}

\def\nd{\mathrm{d}}

\def\snatural{\text{\scalebox{0.8}{$\natural$}}}
\def\ssharp{\text{\scalebox{0.8}{$\sharp$}}}
\def\sflat{\text{\scalebox{0.8}{$\flat$}}}

\def\lp {\left (}
\def\rp {\right )}

\def\Voronoi{Vorono\" \i \hskip 2.5 pt }

\renewcommand{\Im}{{\mathrm{Im} }}

\def\shskip{\hskip 1pt}

\makeatletter
\g@addto@macro\normalsize{\setlength\abovedisplayskip{3pt}}
\makeatother

\makeatletter
\g@addto@macro\normalsize{\setlength\belowdisplayskip{3pt}}
\makeatother

\newcommand{\delete}[1]{}

\theoremstyle{plain}

\newtheorem{thm}{Theorem}[section] \newtheorem{cor}[thm]{Corollary}
\newtheorem{lem}[thm]{Lemma}  \newtheorem{prop}[thm]{Proposition}

 \newtheorem{defn}[thm]{Definition}

\newtheorem {rem}[thm]{Remark}

\numberwithin{equation}{section}

\begin{document}

	\title{Beyond the Weyl barrier for $\mathrm{GL} (2)$ exponential sums   }
	\author[  R. Holowinsky, R. Munshi, and Z. Qi]{Roman Holowinsky, Ritabrata Munshi, and Zhi Qi}

	\address{Department of Mathematics, The Ohio State University\\ 231 W 18th Avenue\\
		Columbus, Ohio 43210, USA}
	\email{holowinsky.1@osu.edu}
	
	\address{Stat-Math Unit, Indian Statistical Institute\\ 203 B.T. Road\\ Kolkata, 700108, India.}
	\email{ritabratamunshi@gmail.com}
	
	\address{School of Mathematical Sciences, Zhejiang University, Hangzhou, 310027, China}
	\email{zhi.qi@zju.edu.cn}
	
	\thanks{The third author is supported by the National Natural Science Foundation of China (Grant No. 12071420).}

	\begin{abstract}
		In this paper, we use the Bessel $\delta$-method, along with new variants of the van der Corput method in two dimensions, to prove non-trivial bounds for $\mathrm{GL}(2)$ exponential sums beyond the Weyl barrier. More explicitly, for sums of $\mathrm{GL}(2)$ Fourier coefficients twisted by $e(f(n))$, with length $N$ and phase $f(n)=N^{\beta} \log n / 2\pi$ or $a n^{\beta}$, non-trivial bounds are established for $ \beta < 1.63651... $, which is beyond the Weyl barrier at $\beta = 3/2$.
	\end{abstract}

	
	\maketitle

	\section{Introduction}

	Let $g \in S^{\star}_k (D, \xi)$ be a holomorphic cusp newform of level $D$, weight $k$, nebentypus character $\xi$, with the Fourier expansion  
	$$g (z) = \sum_{n=1}^{\infty} \lambdaup_g (n) n^{(k-1)/2} e (n z), \quad e(z) = e^{2 \pi i z} ,$$
	for  $\Im \, z > 0$.  
	
	In this paper,  we consider the following smoothed exponential sum 
	\begin{align}
		S_{f} (N) = \sum_{n=1}^\infty \lambdaup_g(n) e(f  (n) ) V\left(\frac{n}{N}\right),
	\end{align}
	where the weight function $V \in  C_c^{\infty} (0, \infty) $ and the phase function $f  $ is of the  form:
	\begin{align}\label{0eq: f(x)=Tphi(x/N)}
		f (x) = N^{\beta} \phi (x/N),
	\end{align}
	with  $ \beta > 1+\vepsilon$ for an arbitrarily small $\vepsilon > 0$, and 
	\begin{align}\label{0eq: phi =}
		\phi  (x) = \left\{ \begin{aligned}
			&  \ds   \frac {  \log x}   {2\pi} ,   \\
			& a x^{\shskip\beta}, 
		\end{aligned}  \right.
	\end{align}  
	for a {\it fixed}  real number $a \neq 0$.  For the logarithm case, if one lets $N = t^{1/\beta}$,  then  $S_{f} (N) = N^{- it} \cdot S_g (N, t) $ with 
	\begin{align}
		S_g (N, t) = \sum_{n=1}^\infty  \lambdaup_g(n) n^{i t} V\left(\frac{n}{N}\right). 
	\end{align}
	For the monomial case,   $\beta$ is also considered {\it fixed}, and $S_{a, \shskip \beta} (N) $ is often used to denote the exponential sum:
	\begin{align}\label{1eq: defn of S ab(N)}
		S_{a, \shskip \beta} (N) = \sum_{n=1}^\infty \lambdaup_g(n) e  ( a n^{\shskip\beta}  ) V\left(\frac{n}{N}\right). 
	\end{align}
	
	Thanks to the Rankin--Selberg theory, we know that  $|\lambdaup_g (n)|$'s obey the Ramanujan conjecture on average:
	\begin{align}\label{2eq: Ramanujan}
		\sum_{n \shskip \leqslant N} |\lambdaup_g (n)|^2 \Lt_{g} N. 
	\end{align} 
	An application of the Cauchy--Schwarz inequality followed by \eqref{2eq: Ramanujan} yields the trivial bound $S_f (N) \Lt_{ g} N$. 
	
	In \cite{AHLQ-Bessel},  for the range $1 - \vepsilon < \beta < 3/2 - \vepsilon$, with the aid of a so-called   Bessel $\delta$-method, the following non-trivial `Weyl bound' is proven:
	\begin{align}\label{0eq: bound for 3/2}
		S_{f} (N) \Lt_{g, \shskip   \phi, \shskip \vepsilon} N^{   \frac 1   2 + \frac 1 3 \beta     + \vepsilon    }. 
	\end{align}
	This extends a result of Jutila \cite{Jut87} for modular forms $g$ of level $D =1$. The primary purpose of this paper is to break the upper `Weyl barrier' at $\beta = 3/2$. For this we have the following theorem.  


	\begin{thm}\label{main-theorem 1}
		
		Let $N  >  1$.  Let $a \neq 0$ be a fixed real number.
		Let $V (x) \in C_c^{\infty} (0, \infty) $ be a smooth function with support in $[1, 2]$ and derivatives $V^{(j)} (x) \Lt_{j} 1$  for every $j = 0, 1, 2, ...$.  Let $g \in S^{\star}_k (D, \xi)$ and $ \lambdaup_g(n) $ be its Fourier coefficients.  
		
		{\rm(1)}  We have  
		\begin{align}\label{1eq: main bound, log, 1}
			\sum_{n=1}^\infty  \lambdaup_g(n) n^{i t} V\left(\frac{n}{N}\right)  \Lt_{g,  \shskip \vepsilon}   N^{  \frac {115} {188} } t^{ \frac {139} {564} +\vepsilon}  ,  
		\end{align}
		if $t^{ \frac {139} {219}  + \vepsilon}   \leqslant N \leqslant t^{\frac {79} {115}  }$, and
		\begin{align}\label{1eq: main bound, log, 2}
			\sum_{n=1}^\infty  \lambdaup_g(n) n^{i t} V\left(\frac{n}{N}\right)  \Lt_{g,  \shskip \vepsilon} N^{  \frac {2763} {3758} } t^{ \frac {304} {1879} +\vepsilon} ,
		\end{align}
		if $t^{ \frac {608} {995}  + \vepsilon} \leqslant N \leqslant t^{\frac  {2791} {4311} }    $. 
		
		{\rm(2)}  We have 
		\begin{equation}\label{1eq: main bound, monomial, 1}
			\sum_{n=1}^\infty \lambdaup_g(n) e  ( a n^{\shskip\beta}  ) V\left(\frac{n}{N}\right) \Lt_{g,\shskip a, \shskip \beta, \shskip \vepsilon}   
			N^{  \frac {115} {188} + \frac {139} {564} \beta +\vepsilon},  
		\end{equation} 
		if    $\beta \in [115/79,  219/139) \smallsetminus  \{ 3/2\} $, and
		\begin{equation}\label{1eq: main bound, monomial, 2}
			\sum_{n=1}^\infty \lambdaup_g(n) e  ( a n^{\shskip\beta}  ) V\left(\frac{n}{N}\right) \Lt_{g,\shskip a, \shskip \beta, \shskip \vepsilon}   N^{  \frac {2763} {3758} + \frac {304} {1879} \beta + \vepsilon},  
		\end{equation} 
		if $ \beta   \in [ 4311/2791  ,  995/608 ) \smallsetminus \left\{ 8/5, 37/23, 66/41, 29/18, 50/31, 21/13, 34/21, 13/8 \right\}$.

	\end{thm}

	Note that $  {115} / {79} = 1.45569... $, $ 219/139 = 1.57554...$, $ 4311/2791 = 1.54461...$, and $995/608=1.63651... $. Therefore the Weyl barrier at $3/2 = 1.5$ is extended to $  1.63651... $.


	Our idea is to use the two-dimensional stationary phase method to transform the off-diagonal sum in the Bessel $\delta$-method to  certain  double exponential sums, and then develop two new   van der Corput  methods of exponent pairs to treat this type of sums with `almost separable' phase.  More explicitly, if $(\kappaup, \lambdaup)$ is such an exponent pair, then we may prove  
	\begin{align}\label{0eq: bound for S(N)}
		S_f (N) \Lt_{g, \shskip   \phi, \shskip \vepsilon} N^{\frac 7 6 - \frac 2 3 (\lambdaup - \kappaup) + \lp \frac 23 \lambdaup - \frac 1 3 \rp \beta + \vepsilon} .
	\end{align}

	\begin{defn}\label{defn: theta-barrier}
		For an exponent pair $(\kappaup, \lambdaup)$ we define its $\beta$-barrier by
		\begin{align}\label{2eq: defn of theta}
			\beta (\kappaup, \lambdaup) =   \frac { 4\lambdaup   - 4 \kappaup -1 } {  4\lambdaup -2}. 
		\end{align}
	\end{defn}
	
	The bound in \eqref{0eq: bound for S(N)} is better than the trivial bound $N$ if and only if  $\beta$ does not exceed the barrier $ \beta (\kappaup, \lambdaup)$, so we seek  $(\kappaup, \lambdaup)$  with  $\beta$-barrier as large as possible. 
	
	By  using the one-dimensional van der Corput method in the trivial manner, one may already extend the Weyl barrier to a $\beta$-barrier  at $59/38 = 1.55263...$. Next, by our first  van der Corput method, the  exponent pair $ (7/188, 327/376) $ yields the $\beta$-barrier at $ 219/139 = 1.57554...$ as above. Further, by our second  van der Corput method, the  exponent pair $ (359/3758, 2791/3758)$  yields the $\beta$-barrier at  $995/608=1.63651... $. See Remark \ref{rem: trivial}, \S\S \ref{sec: simple process ABABA}, \ref{sec: improvements}, \ref{sec: ABABABA}, and \ref{sec: remarks, 2} for   detailed discussions.  
	
	Our secondary object is to improve the `Weyl bound' in \eqref{0eq: bound for 3/2} for $1 < \beta < 3/2$. However, the quantity of   improvement is not our main concern. 
	
	\begin{thm}\label{main-theorem 2}
		
		Let notation be as above. Let $q  $ be a positive integer. Set $Q = 2^q$ and define  $$ \beta_1 = \frac {219} {139}, \qquad  \beta_q = 1 + \frac {9 Q} {7+ 9 q Q} \quad \text{{\rm(}$q=2, 3, ...${\rm)}}.$$
		We have  
		\begin{equation} \label{0eq: bound, q, 2}
			S_{f} (N) \Lt_{g,\shskip \phi, \shskip \vepsilon}   
			N^{ \frac 1 2 + \frac 1 3 \beta + \frac {     7 (2q+7) } { 12 (  27 Q-7)} - \frac {       7 (2q+5)  } { 12 (  27 Q-7)} \beta +\vepsilon}   
		\end{equation} 
		for $\beta \in [\beta_{q+1}, \beta_q)  $, with $\beta \neq 1 + 1/(q+1)$ in the monomial case. 
	\end{thm}
	
	The estimate  in \eqref{0eq: bound, q, 2} is a consequence  of our first van der Corput method, and may be considered as `sub-Weyl' for $1 < \beta < 1.57554...  $. Note that when $q = 1$,  \eqref{0eq: bound, q, 2} amounts to \eqref{1eq: main bound, log, 1} and  \eqref{1eq: main bound, monomial, 1} in Theorem \ref{main-theorem 1}.  Our second  van der Corput method, though   stronger in principle, does not always work for $\beta < 1.54461...$. See    \S  \ref{sec: remarks, 2}.

	Theorem \ref{main-theorem 2} may be further improved by the Vinogradov method if $ \beta  $ is close to $1$. 
	
	\begin{thm}\label{main-theorem: Vinogradov}
		There is an absolute constant $c > 0$ such that 
		\begin{equation} \label{0eq: bound, Vinogradov}
			S_{f} (N) \Lt_{g,\shskip \phi,  \shskip \vepsilon}   
			N^{  \frac 1 2 + \frac 1 3 \beta  -   \frac 1 3 c (\beta-1)^3 / \beta^2 + \vepsilon }    
		\end{equation} 
		for  $1 < \beta \leqslant 4/3$, with $\beta \neq 1+1/q$ {\rm(}$q=3, 4, ...${\rm)} in the monomial case. 
	\end{thm}
	
	Finally, for the non-generic case when $\phi (x) = a x^{1+1/q}$, we can still attain a sub-Weyl bound  by the Weyl method. 
	
	\begin{thm}\label{main-theorem: Weyl}
		Let notation be as above. For $q=2, 3,...$ set $Q= 2^q$. We have
		\begin{align}\label{1eq: bound for S, q odd}
			S_{a, \shskip 1+1/q} (N) \Lt_{a, \shskip q, \shskip \vepsilon} N^{\frac 1 2 + \frac {   Q -1/(q+1)} {3    Q - 2 /(q+2)} \frac {q+1} {q} + \vepsilon  }   
		\end{align}
		if $q$ is odd, and
		\begin{align}\label{1eq: bound for S, q even}
			S_{a, \shskip 1+1/q} (N) \Lt_{a, \shskip q, \shskip \vepsilon} N^{\frac 1 2 + \frac {   Q +1/(q+2)} {3    Q +4  /(q+2)} \frac {q+1} {q} + \vepsilon  }   
		\end{align}
		if $q$ is even. 
	\end{thm}


	\subsection*{Reduction of the sub-Weyl subconvex  problem}
	
	Let $ L (s, g) $ be the $L$-function associated to the holomorphic newform $g$. The functional equation and the Phragm\'en--Lindel\"of
	principle imply the $t$-aspect convex bound
	\begin{align*}
		L (1/2+ it , g) \Lt_{ g, \shskip \vepsilon} t^{  1 / 2+\vepsilon}, \qquad \text{$t>1$} . 
	\end{align*}
	By the approximate functional equation, 
	\begin{align}\label{1eq: AFE}
		L (1/2+it, g) \Lt_{ g, \shskip \vepsilon} t^{\vepsilon} \sup_{t^{\theta} < N < t^{1+ \vepsilon}}  \frac {|S_g (N, t)|} {\sqrt{N}} + t^{\theta / 2 }, 
	\end{align}
	where  $N$ is dyadic.  The Weyl bound in \eqref{0eq: bound for 3/2} reads $ S_g (N, t) \Lt \sqrt{N} t^{1/3+\vepsilon} $. By substituting this into  \eqref{1eq: AFE}  and choosing $\theta = 2/3+\vepsilon$, we obtain the Weyl subconvex  bound:
	\begin{align*}
		L (1/2+ it , g) \Lt_{ g, \shskip \vepsilon} t^{  1 / 3 +\vepsilon}, 
	\end{align*}
	which was first proven by Good \cite{Good} in the full-level case $D = 1$. Any bound of the type
	\begin{align*}
		L (1/2+ it , g) \Lt_{ g, \shskip \vepsilon} t^{  1 / 3 - \rho +\vepsilon}, 
	\end{align*}
with $\rho > 0$, is a sub-Weyl subconvex bound.

	\begin{thm}\label{thm: sub-Weyl reduction}
		For any given $\delta > 0$, there exists $ \rho > 0 $ such that 
		\begin{align}\label{1eq: reduction, sub-Weyl}
			L (1/2+it, g) \Lt_{ g, \shskip \vepsilon} t^{\vepsilon} \sup_{t^{1-\delta} < N < t^{1+ \vepsilon}}  \frac {|S_g (N, t)|} {\sqrt{N}} + t^{  1 / 3 - \rho +\vepsilon},  
		\end{align}
		with $N$ dyadic. 
	\end{thm}
	
	\begin{proof}
		Choose $\theta = 608/995  + \vepsilon$. Then Theorem \ref{main-theorem 1}, \ref{main-theorem 2}, and \ref{main-theorem: Vinogradov} ensure the existence of $\rho > 0$ so that  
		\begin{align}\label{1eq: sub-Weyl for S(N, t)}
			S_g (N, t) \Lt \sqrt{N} t^{1/3 - \rho + \vepsilon},
		\end{align}
		whenever $ t^{\theta} < N \leqslant t^{1 - \delta} $. Thus \eqref{1eq: reduction, sub-Weyl} follows immediately on inserting \eqref{1eq: sub-Weyl for S(N, t)} into \eqref{1eq: AFE}. 
	\end{proof}

Theorem \ref{thm: sub-Weyl reduction} manifests that to get a sub-Weyl subconvex bound for the $L (s, g)$  it suffices to prove sub-Weyl bounds for $S_g (N, t)$ with $N$ in the transition  range $ t^{1-\delta} < N < t^{1+\vepsilon} $.

	\subsection*{Notation} 
	By $F \Lt G$ or $F = O (G)$ we mean that $|F| \leqslant c G$ for some constant $c > 0$, and by $F \asymp G$  we mean that  $F \Lt G$ and $G \Lt F$. We write $F \Lt_{g, \shskip \phi, \, \dots} G$ or $F = O_{g, \shskip \phi, \, \dots} (G)$ if the implied constant $c$ depends on  $g$, $\phi, \dots$. For notational simplicity, in the case  $\phi (x) = a x^{\shskip \beta}$, we shall not put $\phi$, $a$ or $\beta$ in the subscripts of $\Lt$ and $O$. 
	
	Let $p$ always stand for prime. The notation $n \sim N$ or $p \sim P$ is used for integers or primes in the dyadic segment $[N, 2N]$ or $[P, 2P]$, respectively. 
	
	We adopt the usual $\vepsilon$-convention of analytic number theory; the value of $\vepsilon $ may differ from one occurrence to another.


	\section{Setup}
	
	Throughout this paper, we assume $1+\vepsilon < \beta < 5/3$ and set $T = N^{\beta}$, so that
	\begin{align}\label{2eq: N<T}
		N^{1+\vepsilon} < T < N^{5/3 }. 
	\end{align}
	
	We start with the following result from \cite[\S 4]{AHLQ-Bessel}, which is a consequence of applications of  the \Voronoi summation formula along with  the Bessel $\delta$-identity. 
	
	\begin{prop}\label{key lemma}
		Let $U(x), V (x) \in C_c^{\infty} (0, \infty) $ be supported in $[1, 2]$, with $U (x) \geqslant 0$ and $V^{(j)} (x) \Lt_{j} 1$  for every $j = 0, 1, 2...$. Define $C_U = (1+i) / \widetilde{U} (3/4)$, with $\widetilde{U}$ the Mellin transform of $U$. For a fixed newform $g \in S^{\star}_k (D, \xi)$, let $ \lambdaup_g(n) $ be its Fourier coefficients, and let $\eta_g$ denote its Atkin--Lehner pseudo-eigenvalue. 	Let parameters $N, X > 1$, $ P > D$ be such that
		\begin{align}\label{4eq: asummptions on N,X,P}
			P^2/N < X, \qquad N < X^{1-\vepsilon}.
		\end{align} 
		Let $P^{\star}$ be the number of primes in $[P, 2P]$.
		We have
		\begin{equation}\label{4eq: S(N)=S(N,X,P)}
			\begin{split}
				S_{f} (N) = \sum_{n \sim N} \lambdaup _g(n)e(f (n)) V\left(\frac{n}{N}\right) = S_{f} (N, X, P)
				+ O  \left( \frac {P \hskip -1pt \sqrt N } { \sqrt X } + \frac{N^{5/4}X^{1/4}}{P^{3/2}}\right),
			\end{split}
		\end{equation}
		with
		\begin{equation}\label{beginning object}
			\begin{split}
				S_{f} (N, X, P) = \frac{ N^{1/4} }{P^{\star} X^{3/4} }  \sum_{ p \shskip \sim P}   
				\frac {\xi (p)}   {\sqrt {p}}
				& \sum_{r \sim N}e(f (r))  V_{\snatural}   \left(\frac{r}{N}\right) \\
				\cdot &\sum_{n \sim DX} \overline {\lambdaup_{  g } (n)} S(n,r;p)
				e \bigg( \frac {2 \sqrt {n r}} { \sqrt D p } \bigg) U \Big(  \frac { n } { D X }  \Big),
			\end{split}
		\end{equation}
		where $V_{\snatural}    (x) =  C_U \eta_g \xi (-1)  D^{-1/2} \cdot x^{1/4} V (x) $  is again supported in $[1, 2]$, with  $V_{\snatural}^{(j)} (x) \Lt_{j} 1$. 
	\end{prop}

	For convenience, we introduce a parameter $K$ such that 
	\begin{align}\label{5eq: XN=P2K2}
		X = {P^2K^2}/{N},  
	\end{align}
	\begin{align}\label{5eq: K<T}  N^{\vepsilon} < K < T^{1-\vepsilon}, 
	\end{align} 
	\begin{align}\label{2eq: N<PK}
		N^{1+\vepsilon} < PK. 
	\end{align}
	It is clear that the assumptions in \eqref{4eq: asummptions on N,X,P} are well justified. 

	Recall from \eqref{0eq: f(x)=Tphi(x/N)} that $f (x) = T \phi (x/N)$. An application of the Poisson summation to the $r$-sum in \eqref{beginning object} leads us to 
	\begin{align*} 
		S_{f} (N, X, P) = \frac{N^{2 }}{P^{\star} (P K)^{3/2} }   \sum_{n \sim DX}   \overline {\lambdaup_{  g } (n)} 
		U \Big(\frac{n}{D X} \Big) \hskip -1.5pt
		\sum_{p \shskip \sim P} \hskip -1.5pt \frac {\xi (p)} {\sqrt p} \hskip -1.5pt
		\mathop{\sum_{r \shskip \Lt \shskip R }}_{ (r, \shskip p)=1 } \hskip -1.5pt e \hskip -1pt \left(\hskip -1pt -\frac{\overline{r}n }{p} \hskip -1pt \right) \CalJ (n,r,p) & \\
		+ O \big(N^{-A}\big) &, 
	\end{align*}
	where 
	\begin{equation*}
		\begin{split}
			\CalJ (y ) =	\CalJ (y,r,p)=\int_{0 }^{\infty} V_{\snatural}   (x) e \bigg(    T \phi (x) -\frac{N r x}{p} +   \frac{2\sqrt{N x y}}{ \sqrt D p}   \bigg)  \mathrm{d} x ,
		\end{split}
	\end{equation*}
	and
	\begin{align}\label{5eq: R = Pt/N}
		R =    P T / N  .
	\end{align}

	Next, by the Cauchy inequality and the Ramanujan bound on average for  the Fourier coefficients $  {\lambdaup_{  g } (n)} $ as in \eqref{2eq: Ramanujan}, we infer that
	\begin{align*} 
		S_{f} (N,   X,   P)^2  & \Lt_{g}   \frac{N^{3  }}{P^{\star 2}   {PK} }
		\sum_{n \sim DX}\Bigg|  \sum_{p \shskip \sim P} \hskip -1.5pt \frac {\xi (p)} {\sqrt p}   \hskip -1.5pt
		\mathop{\sum_{r \shskip \Lt \shskip R }}_{ (r, \shskip p)=1 } 
		e \hskip -1pt \left(\hskip -1pt -\frac{\overline{r}n }{p} \hskip -1pt \right) 
		\CalJ (n,  r,  p) \Bigg|^2 \hskip -1pt U   \hskip -1pt \Big(\hskip -0.5pt \frac{n}{D X} \hskip -0.5pt \Big) + N^{-A}  \\
		& = \displaystyle\frac{N^{3 }}{{P^{\star 2}}  P K} \hskip -1.5pt \underset{p_1,\shskip p_2 \shskip \sim P}{\sum \sum}  \frac { \xi (p_1 \overline p_2)  }  {\sqrt {p_1 p_2} }  \mathop{\mathop{\sum\sum}_{r_1, \shskip r_2 \shskip \Lt \shskip R }}_{ (r_1,   p_1) = (r_2,   p_2)  =1 }   \\
		& \hskip 41pt \cdot \sum_{n \sim DX} e \hskip -1.5pt \left(\frac{\overline{r}_2 n }{p_2} \hskip -1pt - \hskip -1pt \frac{\overline{r}_1 n }{p_1} \hskip -1pt \right) \hskip -1pt \CalJ (n,r_1,p_1)\overline{\CalJ (n,r_2,p_2)} U \hskip -1pt \Big(\hskip -0.5pt \frac{n}{D X} \hskip -0.5pt \Big) + N^{-A}  . \end{align*}
	After applying Poisson summation with modulus $p_1p_2$ to the $n$-sum, in view of the discussions in \cite[\S\S 5.2, 5.4]{AHLQ-Bessel}, we arrive at (our notation here is slightly different)
	\begin{align}\label{5eq: S<Sdiag+Soff}
		S_{f} (N,   X,   P)^2  \Lt_{g}    \big| {S}_{\mathrm{diag}}^2 \big|    +  \big|   {S}_{\mathrm{off} }^2 \big|   + N^{-A} ,
	\end{align}
	with 
	\begin{equation}\label{S(diag)}
		\begin{split}
			{S}_{\mathrm{diag}}^2  = \frac{N^{3 } X }{P^{\star 2}  P  K }  
			\sum_{p \shskip \sim P}  \frac 1 p \mathop{\mathop{\mathop{\sum\sum}_{r_1, \shskip r_2 \shskip \Lt \shskip R }}_{ (r_1 r_2,  \shskip p)  =1 }}_{r_1\equiv \, r_2 (\mod p)}
			\CalL\left( 0 ; r_1,r_2,p ,p \right)  ,
		\end{split}
	\end{equation}   
	and 
	\begin{equation}\label{S(off)}
		\begin{split}
			{S}_{\mathrm{off}}^2   
			=  \frac{N^{3 } X }{ P^{\star 2}  P      K  }  \, \mathop{\underset{p_1,\shskip p_2 \shskip \sim P}{\sum \sum}}_{p_1 \neq p_2} \frac { \xi (p_1 \overline p_2)  }  {\sqrt {p_1 p_2} }  {\sum_{ 1\leqslant |n| \Lt N/K }}    S (n, p_1, p_2 )  ,
	\end{split}\end{equation}
	in which  
	\begin{align}\label{2eq: S(p)}
		S (n, p_1, p_2)	=   \mathop{\mathop{\mathop{\sum\sum}_{r_1, \shskip r_2 \shskip \Lt \shskip R }}_{ r_1\equiv \shskip \overline{n}  p_2  (\mod p_1) }}_{r_2\equiv -\overline{n}  p_1  (\mod p_2) }   \CalL\left( {DXn} / {p_1p_2}; r_1,r_2,p_1,p_2\right), 
	\end{align}
	where 
	\begin{align*}
		\CalL ( \vw  ) = \CalL ( \vw; r_1,r_2,p_1,p_2  )  = \int_{0 }^{\infty} U(y)  \CalJ (DXy,r_1,p_1) \overline{\CalJ (DXy,r_2,p_2)} \,e\left(- \vw y\right)\mathrm{d}y.
	\end{align*}
	In view of $1\leqslant |n| \Lt N/K$ and  $X = {P^2K^2}/{N}$, it is necessary that
	\begin{align}\label{2eq: x >K2/N}
		\frac {K^2} N \Lt	|\vw | \Lt K. 
	\end{align} 
	The following expression of $\CalL (\vw)$ is established in the proof of \cite[Lemma 5.4 (2)]{AHLQ-Bessel}: 
	\begin{align}\label{2eq: L(x), 2}
		\CalL (\vw) = \frac 1 {\sqrt {|\vw|} } \int_{-N^{\vepsilon}}^{N^{\vepsilon}}  
		\widehat{W}_{\snatural} (\vv) \CalL (\vw; \vv)  \shskip \mathrm{d} \vv + O \big(N^{-A} \big),
	\end{align} 
	where $ \widehat{W}_{\snatural} (\vv) $ is of Schwartz class (the Fourier transform of a certain $ {W}_{\snatural} \in C_c^{\infty} (0, \infty)$), satisfying  $
	\widehat{W}_{\snatural} (\vv) \Lt_A (1 +|\vv|)^{-A}$, 
	and
	\begin{align}\label{2eq: defn of L(x;v)}
		\CalL (\vw; \vv) =  \CalL ( \vw; \vv; r_1,r_2,p_1,p_2  )  =	\iint V_{\snatural} (x_1)  \overline{V_{\snatural}  (x_2)} e  ( T \phi (x_1, x_2; \vw; \vv) )   \shskip \mathrm{d} x_1 \shskip \mathrm{d} x_2, 
	\end{align}
	with phase function
	\begin{align}\label{2eq: f (v1, v2; v)}
		\phi ( x_1, x_2; \vw; \vv ) =   \phi (x_1) - y_1 x_1  -  \phi (x_2) + y_2 x_2  + \delta  (x_1, x_2; \vw; \vv )  ,
	\end{align}
	where
	\begin{align}\label{2eq: theta1, theta2}
		& y_1  = \frac {N r_1} {T p_1} , \quad \quad y_2 = \frac {N r_2} {T p_2}  , \\
		\label{2eq: delta(v1, v2; v)}   \delta   (x_1, x_2; \vw ; \vv)  = & \frac {K^2 P^2 } {T \vw }   \bigg( \frac {\sqrt{x_1}} {p_1}  - \frac   {\sqrt{x_2}} {p_2} \bigg)^2 + \frac {K P \vv} {T \vw} \bigg( \frac {\sqrt{x_1}} {p_1 } - \frac {\sqrt{x_2}} {p_2 } \bigg). 
	\end{align} 
	Note that 
	\begin{align}\label{2eq: bounds for delta}
		\frac{\partial^{j_1+j_2} \delta   (x_1, x_2; \vw ; \vv)}{ \partial x_1^{j_1 } \partial x_2^{ j_2} } \Lt_{j_1, \shskip j_2} \frac {K^2} {T |\vw|}, \quad (x_1, x_2) \in (1/16, 64)^2, \, \vv \in [-N^{\vepsilon}, N^{\vepsilon}]. 
	\end{align}
	In view of \eqref{2eq: x >K2/N}, the condition $ K^2/T  > N^{\vepsilon} $ in  \cite[Lemma 5.4 (2)]{AHLQ-Bessel}  may be weakened into  $ K^2/N > N^{\vepsilon} $, so we only require
	\begin{align}\label{2eq: sqrt N < K}
		K > N^{1/2+\vepsilon}. 
	\end{align}
	
	In \cite{AHLQ-Bessel}, it is proven that 
	\begin{align}\label{2eq: S diag}
		{S}_{\mathrm{diag}}^2  \Lt (KN+T N^{\vepsilon}) \log P , 
	\end{align}   
	and 
	\begin{align}\label{2eq: S off diag}
		{S}_{\mathrm{off}}^2   \Lt \frac {N T} {\sqrt{K}} .
	\end{align} 
	It is impossible to improve upon the bound  for the diagonal sum ${S}_{\mathrm{diag}}^2 $. The bound in \eqref{2eq: S off diag} for the off-diagonal sum ${S}_{\mathrm{off}}^2$ comes from estimating the integral $\CalL  ( \vw )$ by the two-dimensional derivative test. However, if a more careful stationary-phase analysis for  $\CalL  ( \vw  )$ is exploited, the double sum $ S (n, p_1, p_2) $ in \eqref{2eq: S(p)} could be expressed by  two-dimensional exponential sums, and hence there is hope for an extra saving for ${S}_{\mathrm{off}}^2  $. 
	For this, in view of \eqref{2eq: S(p)} and \eqref{2eq: L(x), 2}, we have 
	\begin{align}\label{2eq: S(p1, p2) <}
		S (n, p_1, p_2)  \Lt  \frac {\sqrt{p_1p_2}} {\sqrt{X n}} \int_{-N^{\vepsilon}}^{N^{\vepsilon}}  
		\big|\widehat{W}_{\snatural} (\vv) \big| \left| S (\vv; n, p_1, p_2) \right|   \mathrm{d} \vv + N^{-A} , 
	\end{align}
	with
	\begin{align}\label{2eq: S(v;n,p1,p2)}
		S (\vv; n, p_1, p_2) =	\mathop{\mathop{\mathop{\sum\sum}_{r_1, \shskip r_2 \shskip \Lt \shskip R }}_{ r_1\equiv \shskip \overline{n}  p_2  (\mod p_1) }}_{r_2\equiv -\overline{n}  p_1  (\mod p_2) }   \CalL\left( {DXn}/ {p_1p_2}; \vv;  r_1,r_2,p_1,p_2\right) .
	\end{align}

	\begin{prop}\label{prop: exp sum}    For  $\phi$ given as in {\rm\eqref{0eq: phi =}}, with  $a > 0$, define 
		\begin{align}\label{2eq: defn of psi}
			\psi (y) = \left\{ \begin{aligned}
				&  \ds    \frac{\log y}  {2\pi} ,    \\
				&   b y^{\valpha} ,
			\end{aligned}  \right. \quad \quad \valpha = \frac {\beta} {\beta-1}, \quad    b = \frac 1 {\valpha (a\beta)^{1/(\beta-1)}  } .
		\end{align}  
		We have 
		\begin{align}\label{2eq: S = S/T}
			S (\vv; n, p_1, p_2) = \frac {S_{\psi}^2 ( N, T )} {T }     + O   \lp \frac 1 {N^2} \rp , 
		\end{align}
		with  two-dimensional exponential sum 
		\begin{align}\label{2eq: S(p1,p2)}
			S_{\psi}^2 ( N, T ) = \hskip -1pt  \mathop{\sum\sum}_{\varOmega T/N \leqslant    m_1, \shskip  m_2     \leqslant \varOmega' T/N }  \hskip -1pt e \big(  g (m_1, m_2 )  \big)       V   \lp  \frac {N m_1}   T , \frac {N m_2}   T \rp   ,
		\end{align}
		where $  \varOmega ' >   \varOmega > 0  $ are constants, 
		\begin{align}\label{2eq: h (y1, y2)}
			g (y_1, y_2) = T \psi ( N y_2 /T) - T \psi (N y_1 / T )  + N \omega (N y_1/T, Ny_2/T),  
		\end{align}
		the function $ \omega ( y_1, y_2 )  \in C^{\infty}  [\varOmega, \varOmega']^2 $, with
		\begin{align}\label{2eq: bounds for rho}
			\frac{\partial^{j_1+j_2} \omega  (y_1, y_2 )}{ \partial y_1^{j_1 } \partial y_2^{ j_2} }  \Lt_{j_1, \shskip j_2}  1 , 
		\end{align} 
		and the function $V  (y_1, y_2 )  \in C^{\infty}_c [\varOmega, \varOmega']^2 $, with
		\begin{align}\label{2eq: bounds for V}
			\frac{\partial^{j_1+j_2} V  (y_1, y_2 )}{ \partial y_1^{j_1 } \partial y_2^{ j_2} }  \Lt_{j_1, \shskip j_2}  1  ;
		\end{align}
		the implied constants above are independent on the values of $\vv$, $n$, $p_1$, and $p_2$. 
	\end{prop}
	
	\begin{prop}\label{prop: non-generic}
		Suppose that $\phi (y) = a y^{1 + 1/q}$ with $q = 2, 3, ...$. Then 
		\begin{align}\label{2eq: rho aymp}
			y_1^{ q+2 }   \left|	\frac{\partial^{q+2} g  (y_1, y_2 )}{ \partial y_1^{q+2 }   } \right|  \asymp_{ q }  \left\{  \begin{aligned}
				&  { K^2} / {  |\vw|}, & & \ \text{if $q$ is odd}, \\
				&  { K^4} / T |\vw|^2, & & \ \text{if $q$ is even}.
			\end{aligned}  \right.   
		\end{align}   
	\end{prop}

	\section{Stationary phase lemmas}
	The following two lemmas respectively are consequences (or special cases) of	Theorem 7.7.1 and 7.7.5 in two dimensions in H\"ormander's book \cite{Hormander}. In the following, we use the standard abbreviations $\partial_1 = \partial / \partial x_1$ and $\partial_2 = \partial / \partial x_2$. 
	
	\begin{lem} \label{lem: partial integration}
		Let $K \subset \BR^2$ be a compact set and $X$ be an open neighbourhood of $K$. Let $k  $ be a non-negative integer.  If $u  \in C^{k}_c (K)$, $f   \in C^{k+1} (X)$, and $f  $ is real valued, then for $\lambdaup > 0$ we have
		\begin{equation*}
			\begin{split}
				\left| \int_{K} \hskip -3 pt  u(x) e (\lambdaup f (x)) \nd x     \right|  \leqslant \frac {C} {\lambdaup^{  k}}  \sum_{j_1+j_2 \leqslant k } \sup    \frac {\big| \partial_{1}^{j_1} \partial_{2}^{j_2}  u  \big|}    { \shskip | \shskip f'  |^{2k-j_1-j_2}  } ,
			\end{split}
		\end{equation*}
		where   $C$ is bounded when $f$ stays in a bounded set in $C^{k+1} (X)$. 
		
	\end{lem}
	
	\begin{lem} \label{lem: stationary phase}
		Let $K \subset \BR^2$ be a compact set and $X$ be an open neighbourhood of $K$. Let $u   \in C^{4}_c (K)$ and $f  \in C^{7} (X)$. Suppose that $f  $ is real valued. If $  f (x_{{\oldstylenums{0}}}    ) = 0$, $ f' (x_{{\oldstylenums{0}}}    ) = 0$, $\det f'' (x_{{\oldstylenums{0}}}    ) \neq 0$ and $ f' (x) \neq 0$ in $K \smallsetminus \{x_{{\oldstylenums{0}}}     \}$, then for $\lambdaup > 0$ we have
		\begin{equation*}
			\begin{split}
				\left| \hskip -1pt \int_{K} \hskip -3 pt  u(x) e (\lambdaup f (x)) \nd x  \hskip -1pt - \hskip -1pt \frac {  u(x_{{\oldstylenums{0}}}    ) e (\lambdaup f(x_{{\oldstylenums{0}}}    ) ) }  { \lambdaup \hskip -1pt \sqrt{- \det \hskip -1pt f'' (x_{{\oldstylenums{0}}}    )} }    \right|  \hskip -1pt  \leqslant  \hskip -1pt \frac {C} {\lambdaup^2} \bigg(  \hskip -1pt 1 + \frac {1} {  |\det f'' (x_{{\oldstylenums{0}}}    )  |^3 } \hskip -1pt \bigg)  \hskip -4pt \sum_{j_1+j_2 \leqslant 4  } \hskip -1pt  \sup  \big| \partial_{1}^{j_1} \partial_{2}^{j_2} u   \big| ,
			\end{split}
		\end{equation*}
		where   $C$ is bounded when $f (x) $ stays in a bounded set in $C^7 (X)$ and $|x-x_{{\oldstylenums{0}}}    | / |f'(x)|$ has a uniform bound.
		
	\end{lem}
	
	\begin{proof}
		Apply  Theorem 7.7.5  in  \cite{Hormander} with $k=2$ and estimate $L_1 u$ according to its description therein. 
	\end{proof}

	\section{Basic analytic lemmas}\label{sec: analytic lemmas}

	In this section, we prove some simple analytic lemmas which will be used for analyzing the stationary point in \S \ref{sec: analysis of stationary point} and  also the phase functions in the $B$-processes of the two van der Corput methods in \S \S  \ref{sec: simple vdC}, \ref{sec: second vdC}. For simplicity, we shall not be concerned here the domains of functions, as long as they are  defined on compact subsets of $\BR$ or $\BR^2$.

	We start with   Fa\`a di Bruno's formula (see \cite{Faa-di-Bruno}) and its  two-dimensional generalization  in a less precise form. 
	
	\begin{lem}\label{lem: Bruno, 1}
		For smooth functions $ f (x) $  and $x (y)$ we have
		\begin{align*}
			\frac {\nd^j f (x(y))} {\nd y^{  j}} = j ! \sum f^{(k)} (x(y))  \prod_{  i  = 1 }^{ j } \frac {\big(x^{( i )} (y) / i ! \big)^{b_i} } { b_i ! }
		\end{align*}
		where the sum is over all different solutions in non-negative integers $b_1, ..., b_j$ of the equation $ \sum i b_i  = j$, and $k = \sum b_i $. 
	\end{lem}
	
	\begin{lem}\label{lem: Bruno}
		For a smooth composite function $ f  (x_1 (y_1, y_2) , x_2 (y_1, y_2)  )$, its derivative $ (\partial/\partial y_1)^{j_1} (\partial/\partial y_2)^{j_2} f  (x_1 (y_1, y_2) , x_2 (y_1, y_2)  )  $, with $j_1+j_2 > 0$, is a linear combination of 
		$$ 
		\partial_{1}^{\shskip k_1}  \partial_{2}^{\shskip k_2} f  (x_1 (y_1, y_2) , x_2 (y_1, y_2) ) \prod_{  m   = 1 }^{k_1} \partial_1^{i_{ 1 m    } } \partial_2^{i_{2 m  } } x_1 (y_1, y_2)
		\prod_{  n = 1 }^{k_2} \partial_1^{j_{1 n} } \partial_2^{j_{2 n} } x_2 (y_1, y_2)   , 
		$$ 
		for $\sum  i_{1 m     } + \sum  j_{1 n} = j_1$ and  $  \sum  i_{2 m } + \sum  j_{2 n} = j_2$, with $i_{1m} + i_{2m}, j_{1m} + j_{2m} > 0$.  For the  two terms with $ (k_1, k_2 ) =  (0, 1) , (1, 0)$,  the coefficients are equal to $1$. For the term 
		with $ (k_1, k_2) = (j_1, j_2) $, $i_{1m} = j_{2n}= 1$, and $i_{2m} = j_{1n} = 0$,  
		the coefficient is also equal to $1$.
	\end{lem}

	
	Firstly, we have a simple result by Fa\`a di Bruno's formula as follows. 
	
	\begin{lem}\label{lem: 1-dim}
		Let $g (y), \delta (y) $ be smooth functions, with $ \delta^{(j)} (y) \Lt_j \delta $. Then the function
		\begin{align}\label{4eq: rho = , dim 1}
			\rho (y) = g (y + \delta (y) )  - g (y)  
		\end{align}
		has bounds $\rho^{(j)} (y) \Lt_{j, \shskip g} \delta $. 
	\end{lem}
	
	\begin{proof}
		Take the $j$-th derivative on \eqref{4eq: rho = , dim 1} and expand $ (\nd/\nd y)^j g  (y + \delta (y) ) $ by  Fa\`a di Bruno's formula as in Lemma \ref{lem: Bruno, 1}. In view of the bounds for $\delta^{(j)} (y)$, we infer that 
		\begin{align*}
			\rho^{(j)} (y) & = (1+ \delta' (y))^j \cdot g^{(j)} (y + \delta (y) )   - g^{(j)} (y) + O_{j} (\delta) \\
			& = g^{(j)} (y + \delta (y) )   - g^{(j)} (y) + O_{j} (\delta) , 
		\end{align*}
		and  the bound $ \rho^{(j)} (y) \Lt_j \delta $  is clear from  the mean value theorem. 
	\end{proof}
	
	Lemma \ref{lem: 1-dim} may be generalized if  Taylor's theorem is used instead of the mean value theorem. 
	
	\begin{lem}\label{lem: 1-dim, 2}
		Let $g (y), \delta (y) $ be smooth functions, with $ \delta^{(j)} (y) \Lt_j \delta $. Then the function 
		\begin{align}\label{4eq: rho = , dim 1, 2}
			&	\rho_2 (y) = g (y + \delta (y) )  - g (y)  - g' (y) \delta (y) 
		\end{align}
		has bounds $\rho_2^{(j)} (y) \Lt_{j, \shskip g} \delta^2 $. 
	\end{lem}
	
	Moreover, Lemma \ref{lem: 1-dim}  has a  two-dimensional analogue.
	
	\begin{lem}\label{lem: 2-dim}
		Let $g (y_1, y_2)$, $\delta_1 (y_1, y_2) $, and  $\delta_2 (y_1, y_2) $ be smooth functions, with $\partial_1^{j_1} \partial_{2}^{j_2} \allowbreak \delta_1 (y_1, y_2), \, \delta_2 (y_1, y_2) \Lt_{j_1, \shskip j_2} \delta $. Then the function 
		\begin{align}\label{4eq: rho = , dim 2}
			\rho (y_1, y_2) = g (y_1 + \delta_1 (y_1, y_2), y_2 + \delta_2 (y_1, y_2) )  - g (y_1, y_2)  
		\end{align}
		has bounds $\partial_1^{j_1} \partial_{2}^{j_2} \rho (y_1, y_2) \Lt_{j_1, \shskip j_2, \shskip g} \delta $. 
	\end{lem}
	
	Our problem is to solve equations of the form:
	\begin{equation}\label{4eq: equations} 
		f_1 (x_1) = y_1 +   \delta_1 ( x_1, x_2), \qquad   f_2 (x_2) = y_2 + \delta_2 ( x_1, x_2) , 
	\end{equation}
	where $f_1 (x)$, $f_2 (x)$, $ \delta_1 ( x_1, x_2)$, and $\delta_2 ( x_1, x_2)$ are smooth functions satisfying 
	\begin{align}\label{4eq: bounds for f' }
		f_1' (x), \, f_2'(x) \asymp 1,  
	\end{align} 
	\begin{align}\label{4eq: bounds for f}  
		f_1^{(j)} (x), \, f_2^{(j)}(x) \Lt_{j} 1, 
	\end{align}
	and
	\begin{align}\label{4eq: bounds for delta}
		\partial_{1}^{j_1} \partial_{2}^{j_2} \delta_1  (x_1, x_2 )  \Lt_{j_1, \shskip j_2} \delta_1 , \qquad   \partial_{1}^{j_1} \partial_{2}^{j_2} \delta_2   (x_1, x_2 )  \Lt_{j_1, \shskip j_2} \delta_2.  
	\end{align}
	Let $ \delta_1 , \delta_2 \Lt 1 $ be very small compared to the implicit constants in \eqref{4eq: bounds for f' }. Let $x_{{\oldstylenums{0}}1} (y)$ and $x_{{\oldstylenums{0}}2} (y)$ be the inverse of $f_1 (x)$ and  $f_2 (x)$ respectively.  There is a unique solution of \eqref{4eq: equations} which may be written in the form: 
	\begin{align}\label{3eq: defn of x0}
		x_{{\oldstylenums{0}}1}     (y_1, y_2) = x_{{\oldstylenums{0}}1}     (y_1) + \rho_1 (y_1, y_2), \quad x_{{\oldstylenums{0}}2}     (y_1, y_2) = x_{{\oldstylenums{0}}2}     (y_2) + \rho_2  (y_1, y_2).  
	\end{align}
	Note that the uniqueness is obvious because either $f_1 (x_1) -  \delta_1 (x_1, x_2)$ or  $f_2 (x_2) -   \delta_2 (x_1, x_2)$ is monotonic  along any given direction, and that   $\rho_1 (y_1, y_2)$ and $\rho_2  (y_1, y_2)$ are smooth by the implicit function theorem.  Moreover, observe that the solution would simply be $(x_{{\oldstylenums{0}}1}     (y_1), x_{{\oldstylenums{0}}2}     (y_2))$ when $\delta_1  ( x_1, x_2) = \delta_2  ( x_1, x_2)  \equiv 0$. 
	
	\begin{lem}\label{lem: rho small}
		We have 
		\begin{align}\label{3eq: small rho}
			\partial_{1}^{j_1} \partial_{2}^{j_2} \rho_1   (y_1, y_2 ) \Lt \delta_1 , \qquad  \partial_{1}^{j_1} \partial_{2}^{j_2} \rho_2  (y_1, y_2 )  \Lt_{j_1, \shskip j_2}  \delta_2 . 
		\end{align}
	\end{lem}
	
	We first prove a weaker but useful result: 
	
	\begin{lem}\label{lem: x0 < 1}
		We have 
		\begin{align}\label{4eq: v01, v02 bounded}
			\partial_{1}^{j_1} \partial_{2}^{j_2} 	x_{{\oldstylenums{0}}1}     (y_1, y_2), \, \partial_{1}^{j_1} \partial_{2}^{j_2} 	x_{{\oldstylenums{0}}2 } (y_1, y_2) \Lt_{j_1, \shskip j_2}  1.  
		\end{align} 
	\end{lem}
	
	\begin{proof}[Proof of Lemma \ref{lem: x0 < 1}] For this we use an induction on $j_1 +j_2$. 
		The case $j_1 = j_2 =0$ is clear because the domains of our functions are compact. Suppose that \eqref{4eq: v01, v02 bounded} is already proven for $j_1 + j_2 \leqslant j$. For $ j_1 + j_2 = j + 1$, we apply $ (\partial/\partial y_1)^{j_1} (\partial/\partial y_2)^{j_2} $ to both of the equations
		\begin{align}
			\label{4eq: equation 1}	 f_1 (x_{{\oldstylenums{0}}1}     (y_1, y_2) ) -   \delta_1 ( x_{{\oldstylenums{0}}1}     (y_1, y_2) , x_{{\oldstylenums{0}}2}     (y_1, y_2)) & = y_1 , \\
			\label{4eq: equation 2}	 f_2 (x_{{\oldstylenums{0}}2}     (y_1, y_2) ) -   \delta_2 ( x_{{\oldstylenums{0}}1}     (y_1, y_2) , x_{{\oldstylenums{0}}2}     (y_1, y_2)) & = y_2 ,
		\end{align}
		and then use Lemma \ref{lem: Bruno} to expand the left-hand sides. By the induction hypothesis, along with \eqref{4eq: bounds for f} and \eqref{4eq: bounds for delta}, we infer that all the terms in the expansions are $O (1)$ except perhaps for the two with $ (k_1, k_2 ) =  (0, 1) , (1, 0)$.  Therefore we arrive at
		\begin{align*}
			\lp f_1' (x_{{\oldstylenums{0}}1}    ) - \partial_1 \delta_1 ( x_{{\oldstylenums{0}}1}      , x_{{\oldstylenums{0}}2}     ) \rp \partial_{1}^{j_1} \partial_{2}^{j_2} 	x_{{\oldstylenums{0}}1}     -   \partial_2 \delta_1 ( x_{{\oldstylenums{0}}1}      , x_{{\oldstylenums{0}}2}     ) \cdot \partial_{1}^{j_1} \partial_{2}^{j_2} 	x_{{\oldstylenums{0}}2}     & = O (1) , \\
			\partial_1   \delta_2 ( x_{{\oldstylenums{0}}1}      , x_{{\oldstylenums{0}}2}     ) \cdot \partial_{1}^{j_1} \partial_{2}^{j_2} 	x_{{\oldstylenums{0}}1}     +  \lp f_2' (x_{{\oldstylenums{0}}2}    ) - \partial_2  \delta_2 ( x_{{\oldstylenums{0}}1}      , x_{{\oldstylenums{0}}2}     ) \rp \partial_{1}^{j_1} \partial_{2}^{j_2} 	x_{{\oldstylenums{0}}2}     & = O (1) , 
		\end{align*}
		and these together with  \eqref{4eq: bounds for f} and \eqref{4eq: bounds for delta} yield \eqref{4eq: v01, v02 bounded}. 
	\end{proof}
	
	A direct consequence of \eqref{4eq: bounds for delta}, \eqref{4eq: v01, v02 bounded}, and Lemma \ref{lem: Bruno} is the following estimates:
	\begin{align}\label{4eq: delta(v01, v02) small}
		\begin{aligned}
			\frac{\partial^{j_1+j_2} \delta_1 ( x_{{\oldstylenums{0}}1}     (y_1, y_2) , x_{{\oldstylenums{0}}2}     (y_1, y_2)) }{ \partial y_1^{j_1 } \partial y_2^{ j_2} }     \Lt  \delta_1, \quad   \frac{\partial^{j_1+j_2} \delta_2 ( x_{{\oldstylenums{0}}1}     (y_1, y_2) , x_{{\oldstylenums{0}}2}     (y_1, y_2)) }{ \partial y_1^{j_1 } \partial y_2^{ j_2} }      \Lt  \delta_2 . 
		\end{aligned}
	\end{align}

	\begin{proof}[Proof of Lemma \ref{lem: rho small}]
		By symmetry, we only  consider the derivatives of $ \rho_1  (y_1, y_2) $. 
		
		Observe that when applying  $ (\partial/\partial y_1)^{j_1} (\partial/\partial y_2)^{j_2} $  to \eqref{4eq: equation 1}, the right-hand side vanishes if $ j_2 > 0$.  By similar inductive arguments, one can use \eqref{4eq: delta(v01, v02) small} and Lemma \ref{lem: Bruno} to verify 
		\begin{align*} 
			\partial_{1}^{j_1} \partial_{2}^{j_2} 	x_{{\oldstylenums{0}}1}     (y_1, y_2) \Lt_{j_1, \shskip j_2} \delta_1 , 
			\qquad j_2 > 0,  
		\end{align*} 
		which settles the case $j_2 > 0$ since $ \partial_{1}^{j_1} \partial_{2}^{j_2} 	x_{{\oldstylenums{0}}1}     (y_1, y_2) = \partial_{1}^{j_1} \partial_{2}^{j_2} 	\rho_1 (y_1, y_2) $. 
		For the case $ j_2 = 0 $, we use Lemma   \ref{lem: 1-dim}. 
		Since $x_{{\oldstylenums{0}}1}     $ is the inverse of $f_1$, we may rewrite \eqref{4eq: equation 1} as
		\begin{align}\label{5eq: rho = x0 - x0}
			\rho_{1} (y_1, y_2) = x_{{\oldstylenums{0}} 1}     \big(y_1 +   \delta_1 ( x_{{\oldstylenums{0}}1}     (y_1, y_2) , x_{{\oldstylenums{0}}2}     (y_1, y_2)) \big) - x_{{\oldstylenums{0}} 1}     (y_1), 
		\end{align}
		which is exactly in the form of   \eqref{4eq: rho = , dim 1}. Thus the proof of Lemma \ref{lem: rho small} is completed by \eqref{4eq: delta(v01, v02) small} and Lemma \ref{lem: 1-dim}. 
	\end{proof}

	\begin{lem}\label{lem: rho2}
		Suppose that  $\delta_1 = \delta_2 = \delta$. If we write
		\begin{align}\label{5eq: rho2}
			\begin{aligned}
				& \rho_1 (y_1, y_2) = \delta_1 ( x_{{\oldstylenums{0}}1}     (y_1) ,  x_{{\oldstylenums{0}}2}     (y_2) ) \cdot x_{{\oldstylenums{0}} 1}' (y_1) + \rho_{12} (y_1, y_2), \\
				& \rho_2 (y_1, y_2) = \delta_2 ( x_{{\oldstylenums{0}}1}     (y_1) ,  x_{{\oldstylenums{0}}2}     (y_2) ) \cdot x_{{\oldstylenums{0}} 2}' (y_2) + \rho_{22} (y_1, y_2), 
			\end{aligned}
		\end{align} then
		\begin{align}\label{3eq: small rho2}
			\partial_{1}^{j_1} \partial_{2}^{j_2} \rho_{1 2}   (y_1, y_2 ), \, \partial_{1}^{j_1} \partial_{2}^{j_2} \rho_{2 2}   (y_1, y_2 ) \Lt_{j_1, \shskip j_2} \delta^2 . 
		\end{align} 
	\end{lem}
	
	\begin{proof}
		By symmetry, we only  consider the derivatives of $ \rho_{12}  (y_1, y_2) $.	By \eqref{3eq: defn of x0}, \eqref{5eq: rho = x0 - x0},   and \eqref{5eq: rho2}, we split  
		\begin{align*}
			\rho_{1 2}   (y_1, y_2 ) = \rho_{1 2}^{\ssharp}   (y_1, y_2 ) + \rho_{1 2}^{\sflat}   (y_1, y_2 ), 
		\end{align*}
		with 
		\begin{align*}
			\rho_{1 2}^{\ssharp}   (y_1, y_2 ) = & \, x_{{\oldstylenums{0}} 1}      (y_1 +  \delta_1^{\ssharp} (y_1, y_2)   )   - x_{{\oldstylenums{0}} 1}     (y_1) - x_{{\oldstylenums{0}} 1} '    (y_1) \cdot \delta_1^{\ssharp} (y_1, y_2) ,
		\end{align*}
		where $\delta_1^{\ssharp} (y_1, y_2) = \delta_1 ( x_{{\oldstylenums{0}}1}     (y_1, y_2) , x_{{\oldstylenums{0}}2}     (y_1, y_2))$, and
		\begin{align*}
			\rho_{1 2}^{\sflat} \hskip -1pt  (y_1, y_2 ) \hskip -1pt = \hskip -1pt x_{{\oldstylenums{0}} 1} '    (y_1) \big(  \delta_1 ( x_{{\oldstylenums{0}}1}     (y_1) + \rho_1 (y_1, y_2),  x_{{\oldstylenums{0}}2}     (y_2) + \rho_2  (y_1, y_2)  )   - \delta_1 ( x_{{\oldstylenums{0}}1}     (y_1) ,  x_{{\oldstylenums{0}}2}     (y_2)   \big) .
		\end{align*} 
		Since $\rho_{1 2}^{\ssharp}   (y_1, y_2 )$ is of the form \eqref{4eq: rho = , dim 1, 2}, Lemma \ref{lem: 1-dim, 2} may be used to prove $ \partial_{1}^{j_1} \rho_{1 2}^{\ssharp}    (y_1, y_2 )  \Lt \delta^2 $. Next, by applying $  (\partial  /\partial y_2 )^{j_2} $ and Fa\`a di Bruno's formula in Lemma \ref{lem: Bruno, 1},  along with \eqref{4eq: delta(v01, v02) small}, we have 
		\begin{align*}
			\partial_2^{ j_2} \rho_{1 2}^{\ssharp} (y_1, y_2)  =  \big(x_{{\oldstylenums{0}} 1} '     (y_1 +   \delta_1^{\ssharp} (y_1, y_2)   ) - x_{{\oldstylenums{0}} 1} '     (y_1) \big) \cdot   \partial_2^{ j_2} \delta_1^{\ssharp} (y_1, y_2)   + O  ( \delta^2  ) ,
		\end{align*}
		and it follows from Lemma \ref{lem: 1-dim} that $ \partial_{1}^{j_1} \partial_2^{ j_2} \rho_{1 2}^{\ssharp}    (y_1, y_2 )  \Lt \delta^2 $. As for $\rho_{1 2}^{\sflat}    (y_1, y_2 )$ we use Lemma \ref{lem: 2-dim} to get   similar estimates. 
	\end{proof}

	Finally, the following result  in a simplified setting will be useful.
	
	\begin{lem}\label{lem: solution in 1-dimension}
		Let  $f (x)$ and $ \delta  ( x_1, x_2)$ be smooth functions satisfying
		\begin{align*} 
			f ' (x)  \asymp 1, \qquad   f^{(j)}(x) \Lt_{j} 1 , \qquad \partial_{1}^{j_1} \partial_{2}^{j_2} \delta   (x_1, x_2 )  \Lt_{j_1, \shskip j_2} \delta . 
		\end{align*}   
	Let $x_{{\oldstylenums{0}}} (y)$ be the inverse of $f  (x)$. Then the equation
	\begin{equation*}
		f  (x_1 ) = y +   \delta  ( x_1, x_2) 
	\end{equation*}
has a unique solution of the form $x_{\oldstylenums{0}1} (y, x_2) = x_{\oldstylenums{0}} (y) + \rho (y, x_2) $  with  $\partial^{j+j_2}  \rho (y, x_2)/\partial y^{j} \partial x_2^{j_2} \allowbreak \Lt_{j , \shskip j_2} \delta $.
	\end{lem}

	\section{\texorpdfstring{Treating the sum $S (\vv; n, p_1, p_2)$}{Treating the sum S (v; n, p1, p2)}}
	
	
	Since $\vw$ and $\vv$ will play a minor role in what follows, we shall write  $\CalL = \CalL (\vw ; \vv)$, $ \phi ( x_1, x_2 ) =  \phi ( x_1, x_2 ; \vw ; \vv ) $, and $\delta ( x_1, x_2 ) = \delta ( x_1, x_2; \vw ; \vv )$; see \eqref{2eq: defn of L(x;v)}--\eqref{2eq: delta(v1, v2; v)} for their definitions. We stress  that all the implied constants in the sequel will be independent on the values of $\vw$ and $\vv$.  
	
	Recall that  
	\begin{align}\label{5eq: f (x1, x2)}
		\phi ( x_1, x_2 ) =   \phi (x_1) - y_1 x_1  -  \phi (x_2) + y_2 x_2  + \delta  (x_1, x_2 )  . 
	\end{align} 
	Firstly, we have
	\begin{align}\label{4eq: f'=}
		\phi' ( x_1, x_2  ) = ( \phi' (x_1) - y_1,  - \phi' (x_2) + y_2  ) + \delta' ( x_1, x_2  ),
	\end{align}
	and
	\begin{align}\label{4eq: f''=}
		\phi'' ( x_1, x_2  ) = \begin{pmatrix}
			\phi'' (x_1) & \\ & - \phi'' (x_2)
		\end{pmatrix}  + \delta'' ( x_1, x_2  ). 
	\end{align}
	Subsequently, we shall denote $\delta = { K^2} / {T |\vw|} $ and let $\delta$ be sufficiently small. Indeed, it follows from \eqref{2eq: N<T} and \eqref{2eq: x >K2/N} that $\delta \Lt N/T < 1/N^{\vepsilon} $.  It is critical that $\delta   (x_1, x_2 )$ and its  derivatives are very small: 
	\begin{align}\label{4eq: small delta}
		\partial_{1}^{j_1} \partial_{2}^{j_2} \delta   (x_1, x_2 )  \Lt_{j_1, \shskip j_2} \delta , \quad \quad (x_1, x_2) \in (1/16, 64)^2, 
	\end{align}
	as in  \eqref{2eq: bounds for delta}.  
	
	Given \eqref{0eq: phi =}, we have
	\begin{equation}\label{4eq: phi' and phi''}
		\phi' (x) = \left\{ \begin{aligned}
			&  \ds   1 /  {2\pi x} ,   \\
			& a \beta x^{\shskip\beta - 1}, 
		\end{aligned}  \right. \quad \quad \phi'' (x) = \left\{ \begin{aligned}
			&  \ds  - 1 /  {2\pi x^2} ,   \\
			& a \beta (\beta-1) x^{\shskip\beta - 2} .
		\end{aligned}  \right. 
	\end{equation}
	
	In view of \eqref{4eq: f''=}, \eqref{4eq: small delta}, and \eqref{4eq: phi' and phi''}, we have uniformly
	\begin{align}\label{4eq: det f''}
		- \det \phi'' (x_1, x_2)   \Gt 1,   \qquad \partial_{1}^{j_1} \partial_{2}^{j_2} \phi (x_1, x_2) \Lt_{j_1, \shskip j_2}  1,   \qquad (x_1, x_2) \in (1/16, 64)^2. 
	\end{align}
	
	Let $a > 0$. Then  $\phi' (x)$ has  inverse function  
	\begin{align}\label{4eq: v0}
		x_{{\oldstylenums{0}}}     (y) = \left\{ \begin{aligned}
			&  \ds   1 /  {2\pi y} ,  \\
			& (y / a\beta)^{1/(\beta-1)} .
		\end{aligned}  \right.  
	\end{align}  
	
	\delete{	With no loss of generality, we shall also assume $a > 0$. The analysis would be quite easy if  $\delta (x_1, x_2)$ were ignored in $f (x_1, x_2)$; for example,    in the case  $y > 0$ the stationary point would be at $ (x_{{\oldstylenums{0}}}      (y_1), x_{{\oldstylenums{0}}}      (y_2)) $, with
		\begin{align}
			x_{{\oldstylenums{0}}}     (y) = \left\{ \begin{aligned}
				&  \ds   1 /  {2\pi y} ,  \\
				& (y / a\beta)^{1/(\beta-1)} ,
			\end{aligned}  \right.  
		\end{align}  
		while the    stationary phase would be  
		\begin{align}
			(\phi (x_{{\oldstylenums{0}}}     (y_1)) - y_1 x_{{\oldstylenums{0}}}     (y_1)) - (\phi (x_{{\oldstylenums{0}}}     (y_2)) - y_2 x_{{\oldstylenums{0}}}     (y_2))  = - \psi (y_1) + \psi (y_2), 
		\end{align}    with $\psi  $ defined in \eqref{2eq: defn of psi}. }

	\subsection{Application of stationary phase} \label{sec: analysis of stationary point}
	
	It is clear that we are in the setting of \S \ref{sec: analytic lemmas}, but we would like to make the domains of functions more explicit. 
	
	\begin{lem}\label{lem: range of theta}
		Set  $\varTheta = 1/ 4\pi$ or $  a \beta $ and $\varDelta = 2$ or $2^{\beta-1}$ according as $\phi (x) = \log x / 2\pi $ or $ a x^{\shskip \beta} $. Then for any $\varOmega_1, \varOmega_2 \in  ( \varTheta / \varDelta^2 , \varTheta/ \varDelta  )$ and $\varOmega_1', \varOmega_2' \in ( \varDelta^2 \varTheta , \varDelta^3 \varTheta   )$ we have $
		\CalL  = O_{A} \big(T^{-A} \big) $
		for arbitrary $A \geqslant 0$, unless $ (y_1, y_2) \in [   \varOmega_1, c'_1] \times [  \varOmega_2, \varOmega_2'] $. 
	\end{lem}

	\begin{proof}
		Suppose that $(x_1, x_2) \in [1,2]^2$.	Note that the range of $\phi' (x)$ for $x \in [1, 2]$ is   $ [ \varTheta, \varDelta \varTheta]$. Therefore $|\phi' (x_1) - y_1|^2 + |\phi' (x_2) -  y_2|^2 \Gt 1$ for all  $ (y_1, y_2) \notin [ \varOmega_1, \varOmega_1'] \times [ \varOmega_2, \varOmega_2'] $. Thanks to  \eqref{4eq: small delta}, we have  $|\phi' ( x_1, x_2  ) |   \Gt  1$ provided that $ \delta  $ is small enough. It follows from Lemma \ref{lem: partial integration} that the integral $\CalL $ is negligibly small. 
	\end{proof}

	In view of Lemma \ref{lem: range of theta}, we now assume that $ y_1, y_2  \in ( \varTheta / \varDelta^2 , \varDelta^3 \varTheta )$. First of all, we prove that there exists  a unique stationary point of $ \phi (x_1, x_2)$  inside $[1/8, 16]^2$. By \eqref{4eq: f'=}, we need to solve the equations
	\begin{equation}\label{4eq: stationary point equation} 
		\phi' (x_1) = y_1 - \partial_1 \delta ( x_1, x_2), \quad   \phi' (x_2) = y_2 + \partial_2 \delta ( x_1, x_2) .  
	\end{equation}  
	Note that $\phi'$ maps $ [ 1/8, 16] $ onto $ [  \varTheta / \varDelta^3 , \varDelta^4 \varTheta ] $. Thus \eqref{4eq: small delta} implies that $( \varTheta / \varDelta^2 , \varDelta^3 \varTheta )^2$ is contained in the image of the map $(\phi' (x_1)- \partial_1 \delta ( x_1, x_2), \phi' (x_2) + \partial_2 \delta ( x_1, x_2))$, and hence the equations in \eqref{4eq: stationary point equation} are solvable. Moreover, the solution must be unique by our discussions before.  
	Therefore  one may write the stationary point  in the form:
	\begin{align}\label{4eq: stationary point}
		x_{{\oldstylenums{0}}1}     (y_1, y_2) = x_{{\oldstylenums{0}}}     (y_1) - \rho_1 (y_1, y_2), \quad x_{{\oldstylenums{0}}2}     (y_1, y_2) = x_{{\oldstylenums{0}}}     (y_2) + \rho_2  (y_1, y_2), 
	\end{align}
	where $ x_{{\oldstylenums{0}}}      $ is defined as in \eqref{4eq: v0}. By Lemma \ref{lem: rho small} and \ref{lem: x0 < 1}, the functions $x_{{\oldstylenums{0}}1}$, $x_{{\oldstylenums{0}}2}$, $\rho_1$, and $\rho_2$ have the following estimates. 
	
	
	\begin{lem}\label{lem: rho small in stationary}
		For $(y_1, y_2) \in  ( \varTheta / \varDelta^2 , \varDelta^3 \varTheta )^2$ we have 
		\begin{align}\label{6eq: v01, v02 bounded}
			\partial_{1}^{j_1} \partial_{2}^{j_2} 	x_{{\oldstylenums{0}}1}     (y_1, y_2), \, \partial_{1}^{j_1} \partial_{2}^{j_2} 	x_{{\oldstylenums{0}}2 } (y_1, y_2) \Lt_{j_1, \shskip j_2}  1,  
		\end{align} 
		and
		\begin{align}\label{4eq: small rho}
			\partial_{1}^{j_1} \partial_{2}^{j_2} \rho_1   (y_1, y_2 ) , \, \partial_{1}^{j_1} \partial_{2}^{j_2} \rho_2  (y_1, y_2 )  \Lt_{j_1, \shskip j_2}  \delta . 
		\end{align}
	\end{lem}
	
	Moreover, the derivatives of $\rho_1$ and $\rho_2$ have asymptotic formulae as in Lemma \ref{lem: rho2}. 
	
	\begin{lem}\label{lem: rho12, rho22}
		If we let 
		\begin{align}\label{6eq: rho = }
			\begin{aligned}
				& \rho_1 (y_1, y_2) = \partial_1 \delta   ( x_{{\oldstylenums{0}}}     (y_1) ,  x_{{\oldstylenums{0}}}     (y_2) ) \cdot x_{{\oldstylenums{0}} }' (y_1) + \rho_{12} (y_1, y_2), \\
				& \rho_2 (y_1, y_2) =   \partial_2 \delta  ( x_{{\oldstylenums{0}}}     (y_1) ,  x_{{\oldstylenums{0}}}     (y_2) ) \cdot x_{{\oldstylenums{0}}}' (y_2) + \rho_{22} (y_1, y_2),
			\end{aligned}
		\end{align}
		then 
		\begin{align}\label{6eq: small rho2}
			\partial_{1}^{j_1} \partial_{2}^{j_2} \rho_{1 2}   (y_1, y_2 ), \, \partial_{1}^{j_1} \partial_{2}^{j_2} \rho_{2 2}   (y_1, y_2 ) \Lt_{j_1, \shskip j_2} \delta^2 . 
		\end{align} 
	\end{lem}

	Now we apply Lemma \ref{lem: stationary phase} to the integral $\CalL$ as defined by \eqref{2eq: defn of L(x;v)}--\eqref{2eq: delta(v1, v2; v)}.
	
	\begin{lem}\label{lem: asymptotic of L}
		For $(y_1, y_2) \in  ( \varTheta / \varDelta^2 , \varDelta^3 \varTheta )^2$ we have 
		\begin{align}\label{4eq: asymptotic of L}
			\CalL = 	     e \big( T \big( \psi (y_2) - \psi (y_1) + \rho_{\snatural} (y_1, y_2 ) \big) \big) \cdot \frac { V_{\snatural} (y_1, y_2 )  }  { T  } + O   \lp \frac 1 {T^2} \rp  ,
		\end{align}
		where $\psi (y) $ is defined as in {\rm\eqref{2eq: defn of psi}}, 
		\begin{equation}\label{3eq: rho n =}
			\begin{aligned}
				\rho_{\snatural}  (y_1, y_2 ) & =       \phi \big(x_{{\oldstylenums{0}}}    (y_1) - \rho_1 (y_1, y_2 )   \big)  - \phi (x_{{\oldstylenums{0}}}    (y_1)) + y_1  \rho_1 (y_1, y_2 ) \\
				&\, -     \phi \big(x_{{\oldstylenums{0}}}    (y_2) + \rho_2 (y_1, y_2 ) \big)  + \phi (x_{{\oldstylenums{0}}}    (y_2)) + y_2 \rho_2 (y_1, y_2 )  
				\\
				&\, +   \delta  (x_{{\oldstylenums{0}}1}     (y_1, y_2 )  , x_{{\oldstylenums{0}}2}     (y_1, y_2 )  )   
			\end{aligned}
		\end{equation}
		satisfies
		\begin{align}\label{4eq: bounds for rho}
			\partial_{1}^{j_1} \partial_{2}^{j_2} \rho_{\snatural} (y_1, y_2 ) \Lt_{j_1, \shskip j_2}    \frac N T , 
		\end{align} 
		and the function $V_{\snatural} (y_1, y_2 ) $ is smooth and compactly supported, with
		\begin{align}\label{4eq: bounds for V}
			\partial_{1}^{j_1} \partial_{2}^{j_2} V_{\snatural} (y_1, y_2 ) \Lt_{j_1, \shskip j_2}  1  . 
		\end{align}
	\end{lem}
	
	\begin{proof}
		The formula \eqref{4eq: asymptotic of L} follows from direct calculations. To be precise, the stationary phase is equal to 
		\begin{align*}
			\phi ( x_{{\oldstylenums{0}}1}, x_{{\oldstylenums{0}}2} ) =   \phi (x_{{\oldstylenums{0}}1}) - y_1 x_{{\oldstylenums{0}}1}  -  \phi (x_{{\oldstylenums{0}}2}) + y_2 x_{{\oldstylenums{0}}2}  + \delta  (x_{{\oldstylenums{0}}1}, x_{{\oldstylenums{0}}2} ), 
		\end{align*} 
		and its expression of the form in \eqref{4eq: asymptotic of L} is due to \eqref{4eq: stationary point} and
		\begin{align*}
			- \psi (y) = \phi (x_{{\oldstylenums{0}}} (y)) - y  x_{{\oldstylenums{0}}} (y) + c, 
		\end{align*} 
		where $c = \log (2\pi e) / 2 \pi$ or $0$ according as  $\phi (x) = \log x / 2\pi $ or $ a x^{\shskip \beta} $.
		It is routine to prove \eqref{4eq: bounds for rho} by  the estimates in \eqref{4eq: small delta},  \eqref{6eq: v01, v02 bounded},     \eqref{4eq: small rho}, and Lemma \ref{lem: Bruno}. Recall that $\delta = O(N/T)$.  Moreover, we have
		\begin{align*}
			V_{\snatural }  (y_1, y_2 )  = \frac {  V_{\snatural} (x_{{\oldstylenums{0}}1}     (y_1, y_2 ) )  \overline{V_{\snatural}  (x_{{\oldstylenums{0}}2}     (y_1, y_2 )  ) }   }  {   \sqrt{ - \det \phi'' ( x_{{\oldstylenums{0}}1}     (y_1, y_2 ) , x_{{\oldstylenums{0}}2}      (y_1, y_2 ) ) } } .
		\end{align*} 
		Thus \eqref{4eq: bounds for V} readily follows from  \eqref{4eq: det f''}  and \eqref{6eq: v01, v02 bounded}. Finally, we remark that the
		constant implied in the error term $O \big(1/T^2 \big)$ does not depend on $y_1$ or $ y_2$ because of the uniform bounds in \eqref{4eq: det f''}. 
	\end{proof}

	\begin{lem}\label{lem: rho n 2}
		The function $\rho_{\snatural}  (y_1, y_2 )$ defined in {\rm\eqref{3eq: rho n =}} may be written as
		\begin{align}\label{6eq: rho n2}
			\rho_{\snatural}  (y_1, y_2 ) = \delta  (x_{{\oldstylenums{0}}}     (y_1  )  , x_{{\oldstylenums{0}}}     ( y_2 )  ) + \rho_{\snatural}^2  (y_1, y_2 ),
		\end{align}  
		so that $\partial_{1}^{j_1} \partial_{2}^{j_2} \rho_{\snatural}^2 (y_1, y_2 ) \Lt_{j_1, \shskip j_2}   \delta^2$.  
	\end{lem}
	
	\begin{proof}
		By Taylor's theorem, the first line in \eqref{3eq: rho n =} is equal to
		\begin{align*}
			(y_1 - \phi' (x_{{\oldstylenums{0}}}    (y_1))) 
			\rho_1 (y_1, y_2) + O (\delta^2 )   ,
		\end{align*}
		and hence $ O (\delta^2 ) $ because of $  \phi' (x_{{\oldstylenums{0}}}(y)) = y$. 
		Similarly, the second line  is   $ O (\delta^2 ) $. By \eqref{4eq: small delta}, \eqref{4eq: stationary point}, and Lemma \ref{lem: rho small in stationary},  the mean value theorem  implies that the last line is  equal to
		\begin{align*}
			\delta  (x_{{\oldstylenums{0}}}     (y_1  )  , x_{{\oldstylenums{0}}}     ( y_2 )  ) + O (\delta^2). 
		\end{align*}
		It follows that $   \rho_{\snatural}^2 (y_1, y_2 ) = O (\delta^2)$. In general, Lemma \ref{lem: 1-dim, 2} and \ref{lem: 2-dim} may be exploited to prove that the derivatives of $\rho_{\snatural}^2 (y_1, y_2 )$ are   $ O (\delta^2) $. 
	\end{proof}

	\begin{lem}\label{lem: rho n 3}
		The function $\rho_{\snatural}^2  (y_1, y_2 )$   in {\rm\eqref{6eq: rho n2}} may be written as 
		\begin{align}\label{6eq: rho n3} 
			- \frac 1 2  \partial_1 \delta (x_{{\oldstylenums{0}}}     (y_1  ), x_{{\oldstylenums{0}}}     (y_2  ))^2 x_{{\oldstylenums{0}}}' (y_1)   + \frac 1 2 \partial_2 \delta (x_{{\oldstylenums{0}}}     (y_1  ), x_{{\oldstylenums{0}}}     (y_2  ))^2 x_{{\oldstylenums{0}}}' (y_2)   +  \rho_{\snatural}^3  (y_1, y_2 ), 
		\end{align}
		so that $  \partial_{1}^{j_1} \partial_{2}^{j_2} \rho_{\snatural}^3 (y_1, y_2 ) \Lt_{j_1, \shskip j_2}   \delta^3.$
	\end{lem}
	
	\begin{proof}
		The proof is similar to that of Lemma \ref{lem: rho n 2}. Note that $ \phi'' (x_{{\oldstylenums{0}}}(y)) x_{{\oldstylenums{0}}}' (y) = 1 $. It follows from Taylor's theorem and Lemma \ref{lem: rho12, rho22} that  the first line in \eqref{3eq: rho n =} is equal to 
		\begin{align*}
			\frac 1 2 \phi'' (x_{{\oldstylenums{0}}}(y)) \rho_1 (y_1, y_2)^2 + O (\delta^3) =  \frac 1 2 \partial_1 \delta (x_{{\oldstylenums{0}}}     (y_1  ), x_{{\oldstylenums{0}}}     (y_2  ))^2 x_{{\oldstylenums{0}}}' (y_1)  + O (\delta^3) .
		\end{align*}
		Similarly, the second line    is equal to 
		\begin{align*}
			-  \frac 1 2 \partial_2 \delta (x_{{\oldstylenums{0}}}     (y_1  ), x_{{\oldstylenums{0}}}     (y_2  ))^2 x_{{\oldstylenums{0}}}' (y_2)  + O (\delta^3) .
		\end{align*}
		Moreover, by  Taylor's theorem and Lemma \ref{lem: rho12, rho22},  the last line is  equal to
		\begin{align*}
			\delta  (x_{{\oldstylenums{0}}}     (y_1  )  , x_{{\oldstylenums{0}}}     ( y_2 )  )  - \partial_1 \delta (x_{{\oldstylenums{0}}}     (y_1  ), x_{{\oldstylenums{0}}}     (y_2  ))^2 x_{{\oldstylenums{0}}}' (y_1) + \partial_2 \delta (x_{{\oldstylenums{0}}}     (y_1  ), x_{{\oldstylenums{0}}}     (y_2  ))^2 x_{{\oldstylenums{0}}}' (y_2) + O (\delta^3). 
		\end{align*}
		It follows that $   \rho_{\snatural}^3 (y_1, y_2 ) = O (\delta^3)$. For the general case, it   requires some work to extend Lemma \ref{lem: 1-dim, 2} and \ref{lem: 2-dim} to the next order.  
	\end{proof}

	\subsection{Proof of Proposition \ref{prop: exp sum}}
	Combining Lemma \ref{lem: range of theta} and \ref{lem: asymptotic of L}, we deduce that (see \eqref{2eq: theta1, theta2} and \eqref{2eq: S(v;n,p1,p2)}) 
	\begin{align*}
		S (\vv; n, p_1, p_2) = \frac {S_{\psi}^2  (N, T)} {T } 	 + O   \lp \frac 1 {N^2} \rp , 
	\end{align*}
	where 
	\begin{align*}
		S_{\psi}^2  (N, T) = \hskip -1pt \mathop{\mathop{\mathop{\mathop{\sum\sum}_{\varOmega_1 T p_1 /N \leqslant    r_1     \leqslant \varOmega_1' T p_1 /N }}_{\varOmega_2 T p_2 /N \leqslant    r_2     \leqslant \varOmega_2' T p_2 /N}}_{ r_1\equiv \shskip \overline{n}  p_2  (\mod p_1) }}_{r_2\equiv -\overline{n}  p_1  (\mod p_2) } \hskip -1pt e \big(  T \psi_{\snatural} (Nr_1/Tp_1, Nr_2/Tp_2)  \big)       V_{\snatural}  (Nr_1/Tp_1, N r_2/ Tp_2 )   ,
	\end{align*}
	with 
	\begin{align*}
		\psi_{\snatural} (y_1, y_2 )  =   \psi (y_2) - \psi (y_1) + \rho_{\snatural} (y_1, y_2 )  . 
	\end{align*}
	
	Finally, we have to take care of the congruence conditions on $r_1$ and $r_2$. To this end, we simply write 
	\begin{align*}
		r_1 = a_1 + p_1 m_1, \quad r_2 = a_2 + p_2 m_2,
	\end{align*}
	with representatives $a_1 \in (0, p_1)$ and $a_2 \in (0, p_2)$ such that  $
	a_1 \equiv   \shskip \overline{n}  p_2  (\mod p_1) $ and $
	a_2 \equiv   \shskip - \overline{n}  p_1  (\mod p_2) $. For brevity, denote $ \delta_1 = {N a_1} / {T p_1} $ and  $ \delta_2 = {N a_2} / {T p_2} $.  We arrive at the formula \eqref{2eq: S(p1,p2)} in Proposition  \ref{prop: exp sum} upon choosing 
	\begin{align*}
		\varOmega_1 = \varTheta/\varDelta^2 + \delta_1, \quad \varOmega_1' = \varDelta^2 \varTheta  + \delta_1, \quad \varOmega_2 = \varTheta/\varDelta^2 + \delta_2, \quad \varOmega_2' = \varDelta^2 \varTheta  + \delta_2 , 
	\end{align*}
	and letting
	\begin{align*}
		N/T \cdot	\omega  (y_1,  y_2) =   \psi   (y_2 + \delta_2)   - \psi   (y_2)    - \psi  ( y_1 + \delta_1 ) + \psi   (y_1)     + \rho_{\snatural}  (  y_1 + \delta_1,  y_2 + \delta_2 ), 
	\end{align*} 
	\begin{align*}
		V  (y_1,  y_2) = V_{\snatural} ( y_1 + \delta_1,  y_2 + \delta_2). 
	\end{align*}
	Since $ \delta_1, \delta_2 = O (N/T)  $,  \eqref{2eq: bounds for rho} and \eqref{2eq: bounds for V} respectively follow from \eqref{4eq: bounds for rho} and \eqref{4eq: bounds for V} in Lemma \ref{lem: asymptotic of L}.

	\subsection{Proof of Proposition \ref{prop: non-generic}} Let notation be as above. 
	We have 
	\begin{align*}
		g (y_1, y_2) = T \psi_{\snatural} ( Ny_1 /T  + \delta_1 , Ny_2 /T + \delta_2  ). 
	\end{align*}
	Since $\psi_{\snatural} (y_1, y_2 )  =   \psi (y_2) - \psi (y_1) + \rho_{\snatural} (y_1, y_2 )$ and $\psi (y) = b y^{q+1}$ is a polynomial of degree $q+1$ by  \eqref{2eq: defn of psi}, the problem is reduced to proving
	\begin{align}\label{6eq: psi n aymp}
		\left|	\frac{\partial^{q+2} \rho_{\snatural} (y_1, y_2 )}{ \partial y_1^{q+2 }   } \right|  \asymp_{ q }  \left\{  \begin{aligned}
			&  \delta, & & \ \text{if $q$ is odd}, \\
			&  \delta^2, & & \ \text{if $q$ is even}.
		\end{aligned}  \right.   
	\end{align}  
	To this end, we use Lemma \ref{lem: rho n 2} and \ref{lem: rho n 3} to analyze $\rho_{\snatural}  (  y_1 ,  y_2   )$.  Note that $ x_{\oldstylenums{0}} (y) = c y^q $ for $c = 1/ (a+a/q)^q$ by \eqref{4eq: v0}. By  \eqref{2eq: delta(v1, v2; v)}, we have
	\begin{align*}
		\delta  (x_{{\oldstylenums{0}}}     (y_1  )  , x_{{\oldstylenums{0}}}     ( y_2 )  )  = c	\delta    \bigg(          \frac {P^2 y_1 ^q}  {p_1^2  } -   \frac {2 P^2 y_1^{q/2}   y_2^{q/2} } {p_1 p_2} + \frac   {P^2   y_2^q } {p_2^2} \bigg)  + \frac { \sqrt{c} \delta   \vv} {K} \bigg( \frac {P y_1^{q/2} } {p_1 } - \frac {Py_2^{q/2} } {p_2 }  \bigg) ,
	\end{align*}
	if $\vw > 0$, say, and hence
	\begin{align*}
		\frac{\partial^{q+2} \rho_{\snatural}  (  y_1 ,  y_2   ) }{ \partial y_1^{q+2 }   } = - \frac {2 c P^2} {p_1 p_2} \cdot  \delta  \frac{\partial^{q+2} \big(  y_1^{q/2}   y_2^{q/2} \big) }{ \partial y_1^{q+2 }   } + O (\delta N^{\vepsilon} / K + \delta^2). 
	\end{align*}
	by  \eqref{6eq: rho n2} in Lemma \ref{lem: rho n 2}. Therefore   \eqref{6eq: psi n aymp} is clear if $q$ is odd. However, when $q$ is even,  $ \delta  (x_{{\oldstylenums{0}}}     (y_1  )  , x_{{\oldstylenums{0}}}     ( y_2 )  )$ is a polynomial of degree $q/2$, so $\partial_1^{q+2} \rho_{\snatural}   (y_1, y_2 ) = \partial_1^{q+2} \rho_{\snatural}^2  (y_1, y_2 )$. By \eqref{2eq: delta(v1, v2; v)}, we have
	\begin{align*}
		&\partial_1 \delta (x_{{\oldstylenums{0}}}     (y_1  ), x_{{\oldstylenums{0}}}     (y_2  )) = \frac {\delta P} {p_1}  \bigg( \frac {P}{p_1} - \frac {P}{p_2} \frac {y_2^{q/2}} {y_1^{q/2}}\bigg)  + \frac {\delta \vv} {2 \sqrt{c} K} \frac 1 {y_1^{q/2}} , \\
		&\partial_2 \delta (x_{{\oldstylenums{0}}}     (y_1  ), x_{{\oldstylenums{0}}}     (y_2  )) = \frac {\delta P} {p_2}  \bigg( \frac {P}{p_2} - \frac {P}{p_1} \frac {y_1^{q/2}} {y_2^{q/2}}\bigg)  - \frac {\delta \vv} {2 \sqrt{c} K} \frac 1 {y_2^{q/2}}, 
	\end{align*}
	and it follows from \eqref{6eq: rho n3} in Lemma \ref{lem: rho n 3} that 
	\begin{align*}
		\frac{\partial^{q+2}  \rho_{\snatural}^2  (y_1, y_2 ) }{ \partial y_1^{q+2 }   } = - \frac { c q P^4} {2 p_1^2 p_2^2} \cdot  \delta^2  \frac{\partial^{q+2}  (     y_2^{q } / y_1  ) }{ \partial y_1^{q+2 }   } + O (\delta^2 N^{\vepsilon} / K^2 + \delta^3 ). 
	\end{align*}
	Therefore   \eqref{6eq: psi n aymp} is also clear if $q$ is even.

	\section{The van der Corput methods for almost separable double exponential sums} \label{sec: van der Corput}
	
	The exponential sum $S_{\psi}^2 ( N, T )$    in Proposition \ref{prop: exp sum} has  phase function containing a separable main term $ T \psi (N y_2 / T ) - T \psi (N y_1 / T ) $, with $\psi (y) = \log y /2\pi$ or $b y^{\valpha}$, along with a `mixing'  error term $ N \omega (Ny_1/T, Ny_2/T) $---exponential sums of this type will be called {\it almost separable}. Note that $T = M^{\valpha}$ and $N = M^{\valpha-1}$ if we set $M = T/N$. 
	
	In this section, we shall develop two van der Corput methods for almost separable  double exponential sums. They are very much like the method for one-dimensional exponential sums, and in the end we shall reduce the problem to the one-dimensional case as the sum will become seperable after applying processes $A$ and $B$ several times. 
	
	Our first van der Corput method is relatively simple, and we obtain the $A$-process of Srinivasan \cite{Srinivasan-Lattice-3}. Our second method is analogous to the one-dimensional method but in a less user-friendly form. 
	We shall attain the $\beta$-barrier (Definition \ref{defn: theta-barrier}) at $ 1.63651...   $ by the second method, while we only have  $ 1.57554... $  by the first method. However, the second method does not always work for $\beta < 1.54461...$ ($\valpha > 2.83618...$), but the first method works as long as $\beta > 1$.

	The  double exponential sums studied in the literature are usually of monomial phases approximately of form $ A y_1^{\valpha_1} y_2^{\valpha_2} $, and technical difficulties arise because the  Hessian of the phase might be abnormally small after applying the $A$-process; see \cite[\S 2.2]{Kratzel} and \cite[\S 6]{GK-vdCorput}.  However, we shall not encounter this kind of difficulties since  in our case the phase is `almost separable', the Hessian matrix is `almost diagonal', and the domain is rectangular or `almost rectangular'.
	
	\subsection{Review of the one-dimensional van der Corput method} Our main references here are  \cite[\S 2.1]{Kratzel}, \cite[\S 3]{GK-vdCorput}, \cite[\S 5]{Huxley}, and \cite[\S \S 8.3, 8.4]{IK}. However, the reader may find that our setting is not as general as theirs, but it would enable us to simplify their notation and arguments. 
	
	Let $S_{g}^1 (M ) $ (this $g$ is not the modular form) denote an exponential sum of the type
	\begin{align*}
		S_{g}^1 (M ) =    \sum_{c \leqslant m \leqslant d} e (g (m )) , 
	\end{align*}
	where $ [c, d] \subset [\varOmega M, \varOmega' M] $ (for  fixed   $\varOmega' > \varOmega  > 0$) and the phase $g$ is in the function space $ \boldsymbol{\mathrm{F}}_1^{\gamma} ( M, T  ) $ as defined below.
	
	\begin{defn}\label{5defn: class F1(gamma)}
		Let  $ T > M > 1 $ with $T$ large. Let $\gamma$ be real.  Define	$\boldsymbol{\mathrm{F}}_1^{\gamma}  ( M, T ) $ to be the set of  real  functions $ g \in C^{\infty} [c, d]$, with $ [c, d] \subset [\varOmega M, \varOmega' M] $, of the form 
		\begin{align*}
			g (y) = T \psi (y/M),
		\end{align*}
		with 
		\begin{align*}
			\psi (y ) = \left\{ \begin{aligned}
				& b    \log y  + c  + \delta (y ), && \text{ if } \gamma = 0, \\
				& b  y ^{\gamma} + \delta (y ), && \text{ if } \gamma \neq 0,  
			\end{aligned} \right.
		\end{align*}  such that
		\begin{align*}
			\delta^{(j)} (y) \Lt_{ \gamma, j }  1/T^{\vepsilon}  
		\end{align*}
		for $\vepsilon > 0$ and every $   j = 0, 1, 2, ... $, where  $\varOmega' > \varOmega > 0$, $b$, $c$ real, with $b \neq 0$,  are considered as fixed constants. 
	\end{defn}

	\begin{defn}\label{defn: exp pair, 1}
		A pair $(\kappaup, \lambdaup) \in [0,1/2] \times [1/2, 1]$ is called a {\rm(}one-dimensional{\rm)} exponent pair if the inequality 
		\begin{align}\label{4eq: exp pair, 1}
			S^1_g (M) \Lt_{\vepsilon,   \gamma, (\kappaup, \lambdaup)}  M^{  \lambdaup -    \kappaup}  T^{ \kappaup +\vepsilon}
		\end{align}
		holds for  all $g \in \boldsymbol{\mathrm{F}}_1^{\gamma}  ( M, T  )  $ with a finite exceptions of  $\gamma$. We say that $\gamma$ is admissible for $(\kappaup, \lambdaup) $ if it is not in the finite exceptional set. 
	\end{defn} 
	
	\begin{rem}
		We remark that  $\gamma < 1$ is required in {\rm\cite{Kratzel,GK-vdCorput}} {\rm(}actually, any     $\gamma < 1$  is admissible{\rm)}, but we need to start with $\gamma > 5/2$  in our monomial setting.  
	\end{rem}
	
	\begin{rem}\label{rem: not just T > M}
		In view of {\rm\cite[(3.3.4)]{GK-vdCorput}}, the estimate in {\rm\eqref{4eq: exp pair, 1}} holds for any $T > M > 1$ if and only if   
		\begin{align}\label{4eq: exp pair, 1.1}
			S^1_g (M) \Lt_{\vepsilon,   \gamma, (\kappaup, \lambdaup)}  M^{  \lambdaup -    \kappaup}  T^{ \kappaup +\vepsilon} + M/T 
		\end{align} 
		holds	for any $T, M > 1$ {\rm(}clearly, Definition {\rm\ref{5defn: class F1(gamma)}} can be adapted in this general setting{\rm)}. 
	\end{rem}
	
	
	For example, $(1/6, 2/3) = AB (0, 1)$ is an exponent pair, and its exceptional set is $\{1, 2\}$. Moreover, $ \lp   {13} / {84},  {55} / {84} \rp $ is  Bourgain's exponent pair (\cite[Theorem 6]{Bourgain}),  
	obtained from the Bombieri--Iwaniec method along with the decoupling method. 
	For (3.19) and (4.1) in \cite{Bourgain}, in view of the conditions in  \cite[Theorem 1, 3]{Huxley-4}\footnote{It seems that the $3 F^{(4)2} + 4 F^{(3)} F^{(5)}$ in the determinant in \cite[Theorem   3]{Huxley-4} should read  $3 F^{(4)2} - 4 F^{(3)} F^{(5)}$(Huxley's $F$ is our $\psi$); otherwise, $ \gamma = \frac {5-\sqrt{97}} 4 < 1$ would be exceptional, which is certainly not true. } (see also \cite[Theorem 17.1.4, 17.4.2]{Huxley}), requiring that certain combinations of derivatives $\psi^{(j)} $  are non-vanishing, one  needs $\gamma \neq  1, 3/2, 2, 3$ and  $\gamma \neq  1,   2, 5/2, 3, 7/2, 4$, respectively. For (4.2)  in \cite{Bourgain}, Bourgain   uses the exponent pair $\lp 1/9,   {13}/ {18} \rp = ABA^2B(0, 1)$ and one only needs $\gamma \neq 1, 2, 5/2, 3 $.  Therefore $ \lp   {13} / {84},  {55} / {84} \rp $ has exceptional set  $  \{    1, 3/2, 2, 5/2, 3, 7/2, 4  \}$.


	\begin{lem}[$A$-process]
		If $(\kappaup, \lambdaup)$ is an exponent pair, then so is 
		\begin{align}
			A (\kappaup, \lambdaup) = \lp \frac {\kappaup} {2\kappaup + 2},   \frac { \kappaup + \lambdaup + 1} {2\kappaup + 2} \rp . 
		\end{align}
		Moreover, if  $\gamma \neq 1$ and $\gamma - 1 $ is admissible for $  (\kappaup, \lambdaup)$, then  so is $\gamma$ for $A (\kappaup, \lambdaup)$. 
	\end{lem}

	This $A$-process is the so-called Weyl difference by   the van der Corput--Weyl inequality (see \cite[Theorem 2.5]{Kratzel}):
	\begin{align}\label{6eq: Weyl-vdC, 1}
		\sum_{c \leqslant m \leqslant d} e (g (m)) \Lt \frac { M } { \sqrt{H} } + \Bigg\{ \frac {M} {H} \sum_{1 \leqslant h < H}  \sum_{c \leqslant m \leqslant d-h} e (g (m+h) - g(m)) \Bigg\}^{1/2}.  
	\end{align} In the main case, we choose   $H =   M^{ \frac {2\kappaup-\lambdaup+1} {\kappaup + 1} } / T^{ \frac {\kappaup} {\kappaup + 1} }$ ($H$ is not necessarily an integer here, for one may always replace $H$ by $\lfloor H \rfloor$).  Note that  if $g (y) \in \boldsymbol{\mathrm{F}}_1^{\gamma} ( M, T  )$ then $   g(y+ h) - g (y) \in  \boldsymbol{\mathrm{F}}_1^{\gamma-1} ( M, T h / M )$.  

	\begin{lem}[$B$-process]
		If $(\kappaup, \lambdaup)$ is an exponent pair, then so is 
		\begin{align}
			B (\kappaup, \lambdaup) = \lp \lambdaup -\frac 1 2 ,       \kappaup +\frac 1 2 \rp . 
		\end{align}
		Moreover, if  $\gamma \neq 1$ and  $\gamma / (\gamma-1) $ is admissible for $   (\kappaup, \lambdaup)$, then  so is $\gamma$   for $B (\kappaup, \lambdaup)$.  
	\end{lem}
	
	This $B$-process follows from the van der Corput transform (see \cite[Lemma 3.6]{GK-vdCorput} and \cite[Lemma 5.5.3]{Huxley}): 
	\begin{align}
		\sum_{c \leqslant m \leqslant d} e (g (m)) =        \sum_{ a \leqslant n \leqslant b} \frac {e (f (n)  )} {  f_{\snatural} (n)  } + O \lp \frac {M} {\sqrt{T}} + \log (T/M+2) \rp,
	\end{align} 
	where   $f   $ and $f_{\snatural} $ are defined by
	\begin{align*}
		& f (x) = T \phi (Mx / T), & & \hskip -10pt f_{\snatural} (x) = \sqrt{T}/M  \cdot \phi_{\snatural} (Mx/T)  ,   \\
		& \phi (x) = \psi  ( y_{{\oldstylenums{0}}}     (x)) - x y_{{\oldstylenums{0}}}     (x), & & \hskip -10pt \phi_{\snatural} (x) = {\textstyle \sqrt{\psi'' (y_{{\oldstylenums{0}}}    (x))/ i}},  \qquad \psi' ( y_{{\oldstylenums{0}}}     (x)) = x, 
	\end{align*} 
	and  $ [a, b] =   [g' (c), g' (d) ] $ (
	it is not necessary that $g' (c) \leqslant g' (d)$). 
	It is easy to prove that if  $g (y) \in \boldsymbol{\mathrm{F}}_1^{\gamma} ( M, T  )$ then $f (x) \in \boldsymbol{\mathrm{F}}_1^{\gamma / (\gamma -1) } ( T/ M, T )$ (see Lemma \ref{lem: solution in 1-dimension}). 

	\subsection{The simple van der Corput method}\label{sec: simple vdC}
	
	Now we turn to the first simple van der Corput method.

	
	\begin{defn} 
		Let    $  T, M  > 1 $  with $T$ large.  Let $   1/ T^{\vepsilon} > \delta > 0 $.  Let $\gamma $ be real.    
		Define	$\boldsymbol{\mathrm{F}}_2^{\gamma}   (M , T , \delta ) $ to be the set of real  functions $ g \in C^{\infty} (D)$, with rectangle $D = [c_1, d_1] \times [c_2, d_2] \subset  [\varOmega M, \varOmega' M]^2$, of the form
		\begin{align}
			g (y_1, y_2) = T \psi  (y_1 /M  , y_2/M  )  , \qquad \psi (y_1, y_2) = \psi  (y_1  )  -   \psi  (y_2  ) + \rho  (y_1  , y_2  ), 
		\end{align} where $ g  (y) = T  \psi  (y/M ) \in \boldsymbol{\mathrm{F}}_1^{\gamma}  ( M , T   ) $,  and
		\begin{align}\label{4eq: bound for omega}
			\frac {\partial^{j_1+j_2} \rho (y_1, y_2)} { \partial y_1^{j_1} \partial y_2^{j_2} } \Lt_{   \gamma, j_1, j_2}  \delta ,
		\end{align}
		for all  $   j_1, j_2 = 0, 1, 2, ...  $.  
		We say that the double exponential sum 
		\begin{align*}
			S_{g}^2 (M ) =   \sum_{(m_1, \shskip m_2) \shskip \in D} e (g  (m_1, m_2))  
		\end{align*}
		is almost separable if $g \in \boldsymbol{\mathrm{F}}_2^{\gamma}   (M , T , \delta )$. 
	\end{defn}

	\begin{defn}
		\label{defn: exp pair, 2.1} 
		We say that $ (\kappaup, \lambdaup) \in [0,1/2] \times [1/2, 1]$ is a  $\delta$-exponent pair if the estimate
		\begin{align}\label{4eq: exp pair, 2.1}
			S^2_g (M ) \Lt_{\vepsilon,   \gamma, (\kappaup, \lambdaup)}  M^{  2 \lambdaup -   2  \kappaup}  T^{ 2 \kappaup +\vepsilon} 
		\end{align}
		is valid whenever  $T > M $ and $g \in \boldsymbol{\mathrm{F}}_2^{\gamma} ( M , T , \delta  )  $, with  a finite exceptions of  $\gamma$. We say that $\gamma$ is admissible for $(\kappaup, \lambdaup) $ if it is not in the finite exceptional set. 
	\end{defn}
	
	When $ \delta < 1/ T$, it follows from \eqref{4eq: bound for omega} that $$   \frac{\partial^{j_1 + j_2} e ( T   \rho   (y_1 /M , y_2/M  ))}{\partial y_1^{j_1} \partial y_2^{j_2}} \Lt \frac 1 {M^{j_1 + j_2} } $$
	for $j_1, j_2 = 0, 1$, and one may split $e ( T \rho  (y_1 /M , y_2/M  ))$ out as the weight so that the phase $T   \psi  (y_1/M )  - T   \psi  (y_2/M )$ becomes separable. By partial summation, one deduces easily the following lemma. 
	
	\begin{lem}\label{lem: reduction, separable, 1}
		In the case $ \delta < 1/T$, any one-dimensional exponent pair  is a $\delta$-exponent pair.   
	\end{lem}
	
	\begin{rem}\label{rem: trivial}
		By estimating the $m_1$-sum by {\eqref{4eq: exp pair, 1}} and then the $m_2$-sum trivially, we obtain $S^2_g (M) \Lt M^{\lambdaup-\kappaup+1} T^{\kappaup+\vepsilon} $. We therefore consider $\lp \kappaup /2, (\lambdaup+1)/2   \rp$  as the trivial $\delta$-exponent pair coming from a  one-dimensional $ (\kappaup, \lambdaup)$. 
	\end{rem}
	
	As a consequence of \cite[Theorem 2.16]{Kratzel},  we have the following simple estimate:
	\begin{align}\label{6eq: S2 < M}
		S^2_g (M) \Lt \lp T + \frac {M^2} {T} \rp \log T. 
	\end{align}
	
	\begin{lem}[$A_2$-process]\label{lem: A-process, 2.1}
		Suppose that $\kappaup + 3\lambdaup \geqslant 2$. 	If $(\kappaup, \lambdaup)$ is a $\delta$-exponent pair, then so is 
		\begin{align}\label{6eq: A2}
			A_2 (\kappaup, \lambdaup) =	\lp \frac {\kappaup} {4\kappaup + 2},   \frac { 3 \kappaup +  \lambdaup +1 } {4\kappaup + 2} \rp  . 
		\end{align}
		Moreover, if  $\gamma \neq 1$ and $\gamma - 1 $ is admissible for $(\kappaup, \lambdaup)$, then  so is $\gamma$  for $A_2  (\kappaup, \lambdaup)$. 
	\end{lem}
	
	\begin{proof} 
		
		
		By symmetry, we may assume that  $d_1 - c_1 \leqslant d_2 - c_2$. 	Similar to \eqref{6eq: Weyl-vdC, 1},  for $1 \leqslant H \leqslant   d_1-c_1  $ 
		we have 
		\begin{align}\label{6eq: Weyl-vdC, 2.1}
			S^2_g (M) \Lt \frac {M^2} {\sqrt{H}} + \Bigg\{ \frac {M^2} {H} \sum_{1 \leqslant h < H} \left| S^2_g (M; h) \right|  \Bigg\}^{1/2},
		\end{align}
		where 
		\begin{align*}
			S^2_g (M; h)  =  \sum_{(m_1, \shskip m_2) \shskip \in D(h)}  e (g (m_1+h, m_2+h) - g(m_1, m_2)), 
		\end{align*}
		and $D(h) = D \cap (D - (h, h))$. 
		We have $g (y_1+h, y_2+h) - g(y_1, y_2) \in \boldsymbol{\mathrm{F}}_2^{\gamma-1}   (M , T h/M , \delta ) $ if $g(y_1, y_2) \in \boldsymbol{\mathrm{F}}_2^{\gamma}   (M , T , \delta )$. 
		We now split the $h$-sum in \eqref{6eq: Weyl-vdC, 2.1} according to  $ h \leqslant M^2 / T $ and $h > M^2/T$.  In the first case, \eqref{6eq: S2 < M} yields
		\begin{align*}
			\sum_{1\leqslant h \leqslant M^2/T} \left| S^2_g (M; h) \right|  \Lt \sum_{1\leqslant h \leqslant M^2/T} \frac {M^3} {T h} \Lt \frac {M^3 \log M} {T} < M^2 \log M .
		\end{align*}
		In the second case, by applying \eqref{4eq: exp pair, 2.1}  to $S^2_g (M; h)$  we get 
		\begin{align*}
			\sum_{M^2/T < h < H} \left| S^2_g (M; h) \right| \Lt \sum_{M^2/T < h < H} M^{  2 \lambdaup -   2  \kappaup}  (T h /M)^{ 2 \kappaup +\vepsilon} \Lt H^{2 \kappaup + 1} M^{2 \lambdaup -   4  \kappaup} T^{ 2 \kappaup +\vepsilon}. 
		\end{align*}
		Substituting these into \eqref{6eq: Weyl-vdC, 2.1}, we have
		\begin{align*}
			S^2_g (M) \Lt \frac {M^{2+\vepsilon}} {\sqrt{H}} + H^{  \kappaup } M^{  \lambdaup -   2  \kappaup + 1} T^{   \kappaup + \vepsilon} .
		\end{align*}
		We attain the desired bound on choosing $H  = M^{\frac {4\kappaup - 2\lambdaup + 2 } {2\kappaup + 1}} / T^{\frac {2\kappaup} { 2\kappaup + 1}}    $ provided that it does not exceed $d_1 - c_1$. 
		Otherwise, we have  $
		S^2_g (M) \Lt   {M^{2+\vepsilon}} / {\sqrt{d_1 - c_1}}  $ and also the trivial bound $ S^2_g (M) \Lt M (d_1 -c_1) $, so   $$ S^2_g (M) \Lt M^{\frac 5 3 + \vepsilon} = M^{1 + \frac {\lambdaup} {2\kappa+1} } M^{\frac {4\kappaup - 3 \lambdaup + 2} {3 (2\kappaup + 1)} + \vepsilon } \leqslant M^{1 + \frac {\lambdaup} {2\kappa+1} } M^{\frac {5 \kappaup} {3 (2\kappaup + 1)} + \vepsilon }, $$  
		where   $\kappaup + 3\lambdaup \geqslant 2$ is used for the last inequality, and our result follows  if $T > M^{5/3}$. 
		Finally, if $T \leqslant M^{5/3}$, then by \eqref{6eq: S2 < M} we have
		\begin{align*}
			S^2_g (M) \Lt T^{1+\vepsilon} = T^{\frac {  \kappaup + 1} { 2\kappaup + 1 }} T^{\frac {  \kappaup} { 2\kappaup + 1 }+\vepsilon} \leqslant M^{\frac { 5 \kappaup + 5} {3 (2\kappaup + 1) }} T^{\frac {  \kappaup} { 2\kappaup + 1 } +\vepsilon}   \leqslant M^{\frac {  2 \kappaup + \lambdaup + 1} { 2\kappaup + 1  }} T^{\frac {  \kappaup} { 2\kappaup + 1 } +\vepsilon}  ,
		\end{align*}
		where   $\kappaup + 3\lambdaup \geqslant 2$ is used again for the last inequality. 
	\end{proof}
	
	\begin{cor}\label{cor: Aq, 2.1}
		Let $q$ be a positive integer. Set $Q = 2^q$.
		
		{\rm(1)} We have 
		\begin{align}\label{6eq: Aq}
			A_2^q (\kappaup, \lambdaup) = \lp \frac {\kappaup} { 4 (Q-1) \kappaup + Q}, 1 - \frac {q \kappaup -\lambdaup + 1 } {  4 (Q-1) \kappaup + Q} \rp .
		\end{align} 
		
		{\rm(2)} Let $T > M^{q}$.    Define 
		\begin{align}
			H_q = M^{\frac {4 (Q-1) ((q+1) \kappaup -\lambdaup + 1)} { 4 (Q-1) \kappaup + Q }  }  / T^{\frac { 4 (Q-1) \kappaup } { 4   (Q-1) \kappaup + Q} }. 
		\end{align} 
		For $\kappaup + 3\lambdaup \geqslant 2$, in order for $ A_2^q (\kappaup, \lambdaup) $ to be a $\delta$-exponent pair,  it suffices that the estimate in \eqref{4eq: exp pair, 2.1} is valid for any $ \boldsymbol{\mathrm{F}}_2^{\gamma-q}   (M , T h/M^q , \delta ) $ with  $1 \leqslant h < H_q$, $M^{q+1} < T h $, $\gamma \neq 1, 2, ..., q$, and $\gamma - q$ admissible for $(\kappaup, \lambdaup)$.
	\end{cor}
	
	When $q = 1$, \eqref{6eq: Aq} is reduced to \eqref{6eq: A2}, while the statement in (2) is clear from the Weyl differencing step in the proof of Lemma \ref{lem: A-process, 2.1}. It is easy to prove the results for general $q$ by induction. The reader may also find   \eqref{6eq: Aq}  in \cite[Theorem 7]{Srinivasan-Lattice-3}.
	
	

	\begin{lem}[$B$-process]\label{lem: B-process, 2.1}
		Let $\delta < M/T$.  Suppose that  $3 \kappaup + \lambdaup \geqslant 1$ and $ \kappaup + 3 \lambdaup \geqslant 2$.	 If $(\kappaup, \lambdaup)$ is a $\delta$-exponent pair, then so is  $
		B (\kappaup, \lambdaup)$. 
		Moreover, if  $\gamma \neq 1$ and  $\gamma / (\gamma-1) $ is admissible for $   (\kappaup, \lambdaup)$, then so is $\gamma$   for $B (\kappaup, \lambdaup)$.  
	\end{lem}
	
	
	\begin{proof}
	By applying the two-dimensional van der Corput transform as in \cite[Theorem 2.24]{Kratzel} in our setting, we have
		\begin{align}\label{6eq: Poisson}
			\begin{aligned}
				\sum_{(m_1, \shskip m_2) \shskip \in D} e (g  (m_1, m_2)) =    & \sum_{(n_1, \shskip n_2) \shskip \in E} \frac{e (f  (n_1, n_2))} { {f_{\snatural} (n_1, n_2)}}   + O \bigg( \bigg( \frac {M^2}{\sqrt{T}} +   \sqrt{T}  \bigg) \log T   \bigg), 
			\end{aligned}
		\end{align}
		where  $f$ and $f_{\snatural}$ are defined by 
		\begin{align}
		\label{6eq: f =, simple}	&	f (x_1, x_2) = T \phi  (M x_1/T  , M x_2/T )  , \quad f_{\snatural} (x_1, x_2) = T /M^2 \cdot \phi_{\snatural} (M x_1/T  , M x_2/T ) , \\
			& \phi  (x_1, x_2) = \psi (y_{{\oldstylenums{0}}1}) -   x_1 y_{{\oldstylenums{0}}1} - \psi (y_{{\oldstylenums{0}}2}) +  x_2  y_{{\oldstylenums{0}}2} + \rho (y_{{\oldstylenums{0}}1}, y_{{\oldstylenums{0}}2}) , \\
			&  \phi_{\snatural} (x_1, x_2) = {\textstyle \sqrt{- \det \psi'' (y_{{\oldstylenums{0}}1}, y_{{\oldstylenums{0}}2})}},     \\
		\label{6eq: psi =, simple}	&	\psi' (y_{{\oldstylenums{0}}1})    =   x_1 -  \partial_1 \rho (y_{{\oldstylenums{0}}1}    , y_{{\oldstylenums{0}}2}    ), \qquad  \psi' (y_{{\oldstylenums{0}}2}    )      =   x_2 + \partial_2 \rho (y_{{\oldstylenums{0}}1}    , y_{{\oldstylenums{0}}2}     ), 
		\end{align} 
		and  $E$ is the image of $D$ under the map 
		\begin{align*}
			\begin{aligned}
				x_1 =     {T } /{M } \cdot & \lp \psi '  (   {y_1} / {M }  )    +   \partial_1 \rho  (   {y_1} / {M },   {y_2} / {M }  ) \rp, \\
				x_2   =   {T } / {M } \cdot & \lp \psi '  (   {y_2} / {M }  ) -    \partial_2 \rho  (   {y_1} / {M },   {y_2} / {M }  ) \rp .
			\end{aligned} 	
		\end{align*}
		Let $ g \in \boldsymbol{\mathrm{F}}_2^{\gamma}   (M , T , \delta )  $.  Then the same arguments in the proof of Lemma \ref{lem: asymptotic of L} may be applied here to verify that $f \in \boldsymbol{\mathrm{F}}_2^{\gamma/(\gamma -1)}   (T/M , T , \delta )$ and that
		\begin{align*}
			\frac {\partial^{j_1+j_2}} {\partial x_1^{j_1} \partial x_2^{j_2}} \frac 1 { f_{\snatural} (x_1, x_2) } \Lt \frac {M^2} {T} \cdot \frac 1 {(T/M)^{j_1+j_2}} 
		\end{align*}
		for $j_1, j_2 = 0, 1$. Moreover, the domain $E$ is `almost rectangular'---if     $E$ is regularized into the rectangular image of $D$ under the map 
		\begin{align*}
			\begin{aligned}
				x_1 =     {T } /{M } \cdot   \psi '  (   {y_1} / {M }  ), \qquad 
				x_2   =   {T } / {M } \cdot     \psi '  (   {y_2} / {M }  ) ,
			\end{aligned} 	
		\end{align*}  
	  the rounding error is trivially $O ( M^2/T \cdot T/M  ( \delta T/M + 1 ) ) = O (M)$  by our assumption   $ \delta < M/T$. Note that $M <   M^{2\kappaup -2 \lambdaup + 2} T^{2 \lambdaup - 1} $ for $M <T$. 
		On applying  partial summation on the rectangle, along with the bound in \eqref{4eq: exp pair, 2.1}, the sum on the right-hand side of \eqref{6eq: Poisson} is bounded by $$ \Lt \frac {M^2} {T} \lp \frac T M \rp^{2\lambdaup - 2\kappaup} T^{2\kappaup + \vepsilon} = M^{2\kappaup -2 \lambdaup + 2} T^{2 \lambdaup - 1 + \vepsilon}. $$ 
		It is left to consider the error terms in  \eqref{6eq: Poisson}. 
		In the case $T < M^2$, we have $\sqrt{T} < M^2/\sqrt{T}$, and     
		\begin{align*}
			\frac {M^{2} \log T }{\sqrt{T}} < M^{2\kappaup -2 \lambdaup + 2} T^{2 \lambdaup - 1 + \vepsilon }   
		\end{align*}
		if $ T \geqslant M^{  \frac {4 \lambdaup -4\kappaup} {4 \lambdaup-1} } $, while by \eqref{6eq: S2 < M}, we have 
		\begin{align*}
			S^2_g (M) \Lt   T^{1+\vepsilon} = T^{2 - 2 \lambdaup} T^{2 \lambdaup - 1 + \vepsilon } < M^{ \frac { 2   - 2 \lambdaup -6 \kappaup } {4 \lambdaup-1}}  \cdot M^{  2 \kappaup - 2 \lambdaup + 2} T^{2 \lambdaup - 1 + \vepsilon }  
		\end{align*} 
		for $M < T < M^{  \frac {4 \lambdaup -4\kappaup} {4 \lambdaup-1} }  $, and hence the desired bound by $3 \kappaup + \lambdaup \geqslant 1$.  In the case $ T \geqslant M^2$, the error term is dominated by$    \sqrt{T} \log T$, with  
		\begin{align*}
			\sqrt{T} \log T <  M^{2\kappaup -2 \lambdaup + 2} T^{2 \lambdaup - 1 + \vepsilon }
		\end{align*}
		if $ M \geqslant T^{\frac {3 - 4 \lambdaup  } { 4( \kappaup -  \lambdaup + 1) } }$, and we have trivially
		\begin{align*}
			S^2_g (M) \Lt   M^2 = M^{2\kappaup -2 \lambdaup + 2 } M^{  2 \lambdaup - 2\kappaup } < T^{\frac {2 -   \kappaup - 3 \lambdaup } {\kappaup -  \lambdaup + 1}  }  \cdot M^{2\kappaup -2 \lambdaup + 2 } T^{  2 \lambdaup -1}    
		\end{align*}
		if $ M < T^{\frac {3 - 4 \lambdaup  } { 4( \kappaup -  \lambdaup + 1) } }$, and  hence the desired bound by  $ \kappaup + 3 \lambdaup \geqslant 2$.
	\end{proof}

\begin{rem}\label{rem: two vdCorput}
	The reader may find the statement of {\rm\cite[Theorem 2.24]{Kratzel}} very complicated. The proof is by applying  twice the {\rm(}weighted{\rm)} one-dimensional van der Corput transform   in {\rm\cite[Theorem 2.1]{Kratzel}}. In our setting, however, the proof may be effectively simplified if   {\rm\cite[Lemma 5.5.3]{Huxley}} is used along with rectangular regularization. 
\end{rem}
	
	\subsection{Process $A_2^{q + 1} B A_2  B A_2$}\label{sec: simple process ABABA}
	
	In our setting, we start with an almost separable double exponential sum $S^2_g (M)$ of  phase   $g \in \boldsymbol{\mathrm{F}}_2^{\gamma }   (M , M^{\valpha} , 1/M ) $ with $\valpha > 5/2$ and $\gamma = 0$ or $\valpha$. Our strategy for estimating $ S^2_g (M)$ consists of two steps. The first step is to apply the composite process    $A_2^{q + 1} B A_2  B A_2$ so that the $T$-parameter decreases below $ M = 1/\delta$. The second step is to use the observation in  Lemma \ref{lem: reduction, separable, 1} along with the simple   exponent pair $\lp   1 / 6,   2 /3 \rp = AB (0, 1) $.

	{\renewcommand{\arraystretch}{1.5}
		\begin{table*} 
			\capbtabbox{
				\begin{tabular}{ c|c | c }  
					\hline
					Process   &  $ (M, T) $ & $\gamma$  \\
					\hline
					&   $(M, T)$ & $\gamma$ \\   
					$A_2^{q+1}$ & $   (M, Th / M^{q+1}) $ & $\gamma-q-1$ \\ 
					$A_2^{q+1}B$  & $ (Th /M^{q+2}, Th / M^{q+1})$ & $1 + 1/ (\gamma-q-2)$ \\ 
					$A_2^{q+1}BA_2$ &   $  (Th /M^{q+2}, M h' )  $ & $  1/ (\gamma-q-2)$ \\ 
					$A_2^{q+1}BA_2B$    & $  (M^{q+3} h' /T h , M h' ) $ & $  - 1/ (\gamma-q-3)$ \\ 
					$A_2^{q+1} BA_2BA_2$  &  $ (M^{q+3} h' /T h , T h   h''/M^{q+2} )$ & $  -1 - 1/ (\gamma-q-3)$ \\
					\hline 
				\end{tabular}
			}{
				\caption{ } \label{table: act on (M,T), 1.1}
			} 
		\end{table*} 
	}

	{\renewcommand{\arraystretch}{1.5}
		\begin{table*} 
			\capbtabbox{
				\begin{tabular}{ c|c  }  
					\hline
					Process   & $(\kappaup, \lambdaup )$  \\
					\hline
					$A_2^{q+1} BA_2BA_2$ &  $  \lp \text{\Large $\frac {7} {4(27Q-7)}$}, \text{\small$1$}-  \text{\Large $\frac {7 (2q+5)} {8(27Q-7)}$} \rp $       \\  
					$BA_2BA_2$ & {\Large $  \lp \frac {7} {26}, \frac{31}{52} \rp$}    \\ 
					$A_2BA_2$ & {\Large $  \lp \frac {5} {52}, \frac{10}{13} \rp$}   \\ 
					$  BA_2$ & {\Large $  \lp \frac {5} {16}, \frac{9}{16} \rp$}  \\ 
					$A_2$ & {\Large $   \lp \frac {1 } {16}, \frac {13} {16} \rp $}  \\ 
					& {\Large $   \lp \frac {1} {6}, \frac{2}{3} \rp  $}  \\
					\hline 
				\end{tabular}
			}{
				\caption{ } \label{table: exp pairs, 1.1}
			} 
		\end{table*} 
	} 
	
	Table \ref{table: act on (M,T), 1.1} and \ref{table: exp pairs, 1.1}  exhibit  the intermediate paramaters $(M, T)$ and   phase exponent $\gamma$, and the corresponding exponent pairs  in the process of applying $A_2^{q + 1} B A_2  B A_2$. By Corollary \ref{cor: Aq, 2.1}, the $h$, $h'$, and $h''$ in Table \ref{table: act on (M,T), 1.1} satisfy
	\begin{align}
		h < M^{\frac {14 Q- 7   } { 27Q-7  } \lp q -   \valpha + \frac {  7 } {2  }  \rp } , \quad  h     '  <    {h ^{ \frac { 12} {13} }} / {M^{ \frac { 12} {13} \lp q - \valpha    + \frac {29}{12} \rp }} , \quad   h     ''  <  {M^{q -   \valpha + \frac {11}{4} } {h     '}^{\frac {3} {4} }}  / { h  } ,
	\end{align}
	hence the last $T$-parameter
	\begin{align*}
		T h   h''/M^{q+2} < (Mh')^{\frac 3 4} < \big( M^{  \valpha - q - \frac 4 3   } h \big)^{\frac 9 {13}} < M^{ \frac {18 (\valpha - q + 1) Q - 21}  {54 Q - 14} },
	\end{align*}
	and it is less than $M$ if $ \valpha \leqslant q + 2 + 7/18 Q$. Since  $\{1, 2\}$ is the exceptional set for $(1/6, 2/3)$, the last phase exponent is  admissible if $\gamma \notin \left\{ q+5/2, q+8/3, q+3  \right\}$. The condition $ \kappaup + 3\lambdaup \geqslant 2 $ or $ 3 \kappaup +  \lambdaup \geqslant 1 $ in  Lemma \ref{lem: A-process, 2.1} and   \ref{lem: B-process, 2.1} may be easily checked. 
	For Lemma  \ref{lem: B-process, 2.1}, we also have to verify
	\begin{align*}
		1/M < M^{q+2} / T h, \qquad 1/M < T h / M^{q+3} h'. 
	\end{align*}
	For these we only need $  q-3 < \valpha < q + 32/13 + 7/26 Q$.

	\begin{thm}\label{thm: simple vdC}
		Let $q$ be a positive integer. Set $Q = 2^q$. 	Let  $\valpha \in [ q + 1 + 7/9Q,   q + 2 +7/18Q ]$ and  $\gamma \notin       \{ 1, 2, ..., q+2, q+5/2, q+8/3, q+3 \}  $. Then we have the estimate
		\begin{align}\label{6eq: bound for S, 1}
			S^2_g (M) \Lt_{q, \gamma,  \vepsilon}  M^{2 - \frac {7}{54 Q - 14} \lp q - \valpha + \frac 7 2 \rp + \vepsilon}   
		\end{align}
		for any $g \in \boldsymbol{\mathrm{F}}_2^{\gamma }   (M , M^{\valpha} , 1/M ) $. 
	\end{thm}
	
	In the case $q=1$, the $\beta$-barrier   of $A_2^2 B A_2  B A_2 (1/6, 2/3) = (7/188, 327/376)$ (see Definition \ref{defn: theta-barrier}) is at $ 219/139 = 1.57554... $.

	For comparison,  consider the trivial $\delta$-exponent pair as in Remark {\rm\ref{rem: trivial}} obtained from 
	\begin{align*}
		A^{q } \lp \frac {13}{84}, \frac {55}{84} \rp \ = \lp \frac {13} {110 Q - 26}, 1 - \frac {13q+29} {110 Q - 26} \rp.   
	\end{align*}
	We then have 
	\begin{align}\label{6eq: trivial bound for S}
		S^2_g (M) \Lt_{q, \gamma,  \vepsilon}  M^{2 - \frac {13}{110 Q - 26} \lp q - \valpha + \frac {29} {13} \rp + \vepsilon}  . 
	\end{align}
	Some calculations show that {\rm\eqref{6eq: trivial bound for S}} is inferior to {\rm\eqref{6eq: bound for S, 1}} for $\valpha \in [ q + 1 + 7/9Q,   q + 2 +7/18Q ]$. Moreover, for $q=1$, the $\beta$-barrier is at $59/38 = 1.55263...$, and this is smaller than $1.57554...$. 

	\subsection{Remarks on further improvements}
	\label{sec: improvements}
	
	The only reason that we did not use Bourgain's exponent pair is that $ \kappaup + 3 \lambdaup < 2$ if $(\kappaup, \lambdaup) = B A_2 (13/84, 55/84) = (17/55, 123/220)$ and $3 \kappaup +  \lambdaup < 1$ if $(\kappaup, \lambdaup) = A_2 (13/84, 55/84) = (13/220, 89/110)$.  Note that   $\kappaup + 3 \lambdaup = 2$ if $(\kappaup, \lambdaup) = B A_2(1/6, 2/3) = (5/16, 9/16)$ and $3\kappaup +  \lambdaup = 1$ if $(\kappaup, \lambdaup) = A_2 (1/16, 13/16)$.   However, by examine the proofs, it is easy to see that the conditions  $\kappaup + 3 \lambdaup \geqslant 2$ and $3 \kappaup +  \lambdaup \geqslant 1$ in Lemma \ref{lem: A-process, 2.1} and \ref{lem: B-process, 2.1} may be replaced  by $ T \geqslant M^{  \frac {4 \kappaup - 3\lambdaup+2} {3\kappaup } } $ and $ T \geqslant M^{  \frac {4 \lambdaup -4\kappaup} {4 \lambdaup-1} } $ respectively. Then it no longer works for all $T > M$, so the theory needs a revision and becomes less elegant.  At any rate, by using Bourgain's exponent pair, though   slightly, one may improve Theorem \ref{thm: simple vdC}   and raise the $\beta$-barrier to $ 1.57579... $.

	\subsection{The second van der Corput method} \label{sec: second vdC} 
	
	In order to improve the $\beta$-barrier further, we would like to develop  the second  van der Corput method. 
	By necessity, we need more involved notation and conditions.

	\begin{defn}\label{5defn: class F(gamma)}
		Let    $  T_1 >  M_1   > 1 $ and $ T_2 >  M_2   > 1 $ with $T_1, T_2$ large. 
		Let  $  N > 0 $. Let $\gamma $ be real.   
		Define	$\boldsymbol{\mathrm{F}}_2^{\gamma}   (M_1, M_2, T_1, T_2, N  ) $ to be the set of real  functions $ g \in C^{\infty} (D)$, with  rectangle $D = [c_1, d_1] \times [c_2, d_2] \subset  [\varOmega M_1, \varOmega' M_1] \times [\varOmega M_2, \varOmega' M_2 ]$, of the form
		\begin{align}\label{6eq: g(y1, y2)=...}
			g (y_1, y_2) = T_1 \psi_1 (y_1/M_1)  - T_2 \psi_2 (y_2/M_2) + N \omega   (y_1 /M_1 , y_2/M_2 ), 
		\end{align} where $ g_1 (y) = T_1 \psi_1 (y/M_1) \in \boldsymbol{\mathrm{F}}_1^{\gamma}  ( M_1, T_1  ) $, $   g_2 (y) = T_2 \psi_2 (y/M_2) \in \boldsymbol{\mathrm{F}}_1^{\gamma}  ( M_2, T_2  ) $,   and
		\begin{align}\label{4eq: bound for rho}
			\frac {\partial^{j_1+j_2} \omega (y_1, y_2)} { \partial y_1^{j_1} \partial y_2^{j_2} } \Lt_{   \gamma, j_1, j_2}  1 ,
		\end{align}
		for all  $   j_1, j_2 = 0, 1, 2, ...  $.

		We say that the double exponential sum 
		\begin{align*}
			S_{g}^2 (M_1, M_2  ) =   \sum_{(m_1, \shskip m_2) \shskip \in D} e (g  (m_1, m_2))  
		\end{align*}
		is almost separable if 
		\begin{align}\label{6eq: almost separable}  
			N < \min \{ T_1, T_2 \}^{1-\vepsilon} . 
		\end{align}
	\end{defn}

	Subsequently, we shall always assume that  $(\kappaup, \lambdaup)  $ is an exponent pair and that $\gamma$ is admissible in the sense of Definition \ref{defn: exp pair, 1}.

	\begin{defn}
		\label{defn: exp pair, 2} 
		We say that $ (\kappaup, \lambdaup) $ is an exponent pair for $ ( M_1, M_2, T_1, T_2, N  ) $ if the estimate
		\begin{align}\label{4eq: exp pair, 2}
			S^2_g (M_1, M_2) \Lt_{\vepsilon,   \gamma, (\kappaup, \lambdaup)}  (M_1 M_2)^{   \lambdaup -     \kappaup}  (T_1 T_2)^{  \kappaup +\vepsilon} 
		\end{align}
		is valid for all $g \in \boldsymbol{\mathrm{F}}_2^{\gamma} ( M_1, M_2, T_1, T_2, N  )  $ with $\gamma$ admissible. 
	\end{defn}

	When $N \Lt 1$, it follows from \eqref{4eq: bound for rho} that $$ \frac{\partial^{j_1+j_2} e ( N  \omega   (y_1 /M_1, y_2/M_2 ))}{\partial y_1^{j_1} \partial y_2^{j_2}} \Lt \frac 1 {M_1^{j_1} M_2^{j_2}} $$ for $j_1, j_2 = 0, 1$, 
	and the analogue of Lemma \ref{lem: reduction, separable, 1} follows easily. The observation in Remark \ref{rem: not just T > M} is also used here for \eqref{6eq: bound, separable}. 
	
	\begin{lem}\label{lem: reduction, separable}
		If $N \Lt 1$, then $ (\kappaup, \lambdaup) $ is an exponent pair for any $ ( M_1, M_2, T_1, T_2, N  )  $, and we have uniformly 
		\begin{align}\label{6eq: bound, separable}
			S^2_g (M_1, M_2) \Lt_{\vepsilon,   \gamma, (\kappaup, \lambdaup)} \big( M_1^{  \lambdaup -    \kappaup}  T_1^{ \kappaup +\vepsilon} + M_1/T_1 \big) \big(  M_2^{  \lambdaup -    \kappaup}  T_2^{ \kappaup +\vepsilon} + M_2/T_2 \big)  
		\end{align}
		for any $M_1, M_2, T_1, T_2 > 1$. 
	\end{lem}
	
	
	In the Weyl difference, we now use   $(h_1, \pm h_2)$  instead of $(h, h)$ in Lemma \ref{lem: A-process, 2.1}. 
	
	\begin{lem}[$A$-process]\label{lem: A-process, 2.2}
		Define \begin{align}
			\label{6eq: H1, H2}
			H_1 = M_1^{ \frac {2\kappaup-\lambdaup+1} {\kappaup + 1} } / T_1 ^{ \frac {\kappaup} {\kappaup + 1} }, \qquad H_2 = M_2^{ \frac {2\kappaup-\lambdaup+1} {\kappaup + 1} } / T_2 ^{ \frac {\kappaup} {\kappaup + 1} }. 
		\end{align}  Assume that  
		\begin{align}\label{6eq: assumption for A} 
			T_1 > M_1^{2  }, \qquad  
			T_2 > M_2^{2 } ,
		\end{align}  
		If  $(\kappaup, \lambdaup)$ is an exponent pair for any $(M_1, T_1 h_1/M_1 , M_2, T_2  h_2/M_2 , N (h_1/M_1+h_2/M_2)  )$ with $ 1\leqslant h_1 < H_1$ and  $ 1\leqslant h_2 < H_2$, then $A (\kappaup, \lambdaup) $  is an exponent pair for   $( T_1 , T_2 , M_1, M_2 , N )$. 
	\end{lem}
	
	\begin{proof}
		We use the two-dimensional van der Corput--Weyl inequality in \cite[Theorem 2.21]{Kratzel}: 
		\begin{align}\label{6eq: Weyl, 2.2}
			\begin{aligned}
				S  \Lt \frac {M_1 M_2} {\sqrt{H_1 H_2}} + 
				\Bigg\{ \frac {M_1 M_2} {{H_1 H_2}}  \Bigg(   \sum_{1 \leqslant h_1 < H_1}    \sum_{1 \leqslant h_2 < H_2}    \left| S_{+} (h_1, h_2) \right|   +   \left|S_{-} (h_1, h_2) \right| & \\
				\quad  +   \sum_{1 \leqslant h_1<H_1}   \left|S_{1 }(h_1) \right| +     \sum_{1 \leqslant h_2<H_2}  \left|S_{2 } (h_2) \right| & \Bigg)  \Bigg\}^{1/2}, 
			\end{aligned} 
		\end{align}
		for $1\leqslant H_1 \leqslant d_1-c_1$ and  $1\leqslant H_2 \leqslant d_2-c_2$,
		where
		\begin{align*}
			& S = \sum_{(m_1, \shskip m_2) \shskip \in D} e (g  (m_1, m_2)), \\
			& S_{\pm} (h_1, h_2) = \sum_{(m_1, \shskip m_2) \shskip \in D (h_1,  \pm h_2)}   e ( g  (m_1+h_1, m_2 \pm h_2) -  g(m_1, m_2)),  \\
			& S_{1} (h_1 ) = \sum_{(m_1, \shskip m_2) \shskip \in D (h_1,  0)}  e (g  (m_1+h_1 , m_2 ) -  g(m_1, m_2)), \\
			& S_{2} (h_2 ) = \sum_{(m_1, \shskip m_2) \shskip \in D (0 ,  - h_2 )}  e (g  (m_1 , m_2 - h_2 ) -  g(m_1, m_2)) . 
		\end{align*}
		and $D (h_1, \pm h_2) = D \cap (D - (h_1, \pm h_2))$.  
		For $g \in  \boldsymbol{\mathrm{F}}_2^{\gamma} ( M_1, M_2, T_1, T_2, N  )$, the phase function of $ S_{\pm} (h_1, h_2) $ is in $ \boldsymbol{\mathrm{F}}_2^{\gamma-1} ( M_1, M_2, T_1 h_1/M_1, T_2h_2/M_2, N (h_1/M_1 + h_2/M_2) ) $, while the phase  of $S_{1} (h_1)$ or $S_{2} (h_2)$,    viewed as function of $y_1$ or $y_2$, lies in $\boldsymbol{\mathrm{F}}_1^{\gamma-1} (M_1, T_1 h_1/M_1)$ or $\boldsymbol{\mathrm{F}}_1^{\gamma-1} (M_2, T_2 h_2/M_2)$ by \eqref{6eq: almost separable}, respectively. By the assumptions of this lemma, in particular \eqref{6eq: assumption for A},  the estimates \eqref{4eq: exp pair, 2} and \eqref{4eq: exp pair, 1}  are applicable to these sums whenever   $1\leqslant h_1 < H_1 $ and $1 \leqslant h_2 < H_2$.  Therefore
		\begin{align*}
			S_{\pm} (h_1, h_2) \Lt (M_1 M_2)^{   \lambdaup -    2 \kappaup} (h_1 h_2)^{\kappaup} (T_1 T_2)^{  \kappaup +\vepsilon}, 
		\end{align*}
		and
		\begin{align*}
			S_1 (h_1 ) \Lt   M_1^{   \lambdaup -    2 \kappaup} h_1^{\kappaup}  T_1^{  \kappaup +\vepsilon} \cdot M_2, \qquad  S_2 (h_2) \Lt    M_2^{   \lambdaup -    2 \kappaup} h_2^{\kappaup}  T_2^{  \kappaup +\vepsilon} \cdot M_1. 
		\end{align*}
		By substituting these into \eqref{6eq: Weyl, 2.2} and summing up, we have
		\begin{align}\label{6eq: bound for S}
			\begin{aligned}
				S \Lt \frac {M_1 M_2} {\sqrt{H_1 H_2}} & + (M_1 M_2)^{\frac 1 2 \lambdaup - \kappaup + \frac 1 2} (H_1 H_2)^{\frac 1 2 \kappaup} (T_1 T_2)^{\frac 1 2   \kappaup + \vepsilon} \\
				& + \frac {M_2} {\sqrt{H_2}}   M_1^{\frac 1 2 \lambdaup - \kappaup + \frac 1 2}  H_1 ^{\frac 1 2 \kappaup}  T_1^{\frac 1 2   \kappaup + \vepsilon} + \frac {M_1} {\sqrt{H_1}}   M_2^{\frac 1 2 \lambdaup - \kappaup + \frac 1 2}  H_2 ^{\frac 1 2 \kappaup}  T_2^{\frac 1 2   \kappaup + \vepsilon} . 
			\end{aligned}
		\end{align}
		We attain the desired bound on choosing $H_1   $ and $H_2$ as in \eqref{6eq: H1, H2} when $H_1  \leqslant d_1 - c_1$ and  $H_2  \leqslant d_2 - c_2$ are both satisfied. For the rest of the proof,   with abuse of notation, let $H_1$ and $H_2$ be defined as in \eqref{6eq: H1, H2}.  In the case when $H_1 > d_1 - c_1$ and  $H_2  > d_2 - c_2$, we have $S \Lt {(M_1 M_2)^{1+\vepsilon} } / {\sqrt{(d_1 - c_1)(d_2 - c_2)}}$ by \eqref{6eq: bound for S} and       $S \Lt (d_1 - c_1)(d_2 - c_2)$ by trivial estimation, so 
		\begin{align*}
			S \Lt (M_1 M_2)^{  2 / 3 + \vepsilon}. 
		\end{align*}
		This is adequate because
		\begin{align}\label{6eq: >2/3}
			\frac {\kappaup + \lambdaup + 1} { 2\kappaup + 2} \geqslant \frac 2 3, 
		\end{align}
		so that 
		\begin{align*}
			S \Lt (M_1 M_2)^{ \frac {\kappaup + \lambdaup + 1} { 2\kappaup + 2} + \vepsilon} < (M_1 M_2)^{ \frac {  \lambdaup + 1} { 2\kappaup + 2} }  ( T_1 T_2)^{\frac {\kappaup  } { 2\kappaup + 2} + \vepsilon} . 
		\end{align*}
		In the case when $H_1 > d_1 - c_1$ and  $H_2 \leqslant d_2 - c_2$, we have $S\Lt {(M_1 M_2)^{1+\vepsilon}} / {\sqrt{(d_1 - c_1) H_2}}$  by \eqref{6eq: bound for S} and        $S \Lt (d_1 - c_1)M_2^{1+\vepsilon}/\sqrt{H_2}$  by applying \eqref{4eq: exp pair, 1} to the $m_2$-sum, with exponent pair $A(\kappaup, \lambdaup)$, followed by trivial estimation for the $m_1$-sum, so  
		\begin{align*}
			S \Lt  M_1^{2/3 +\vepsilon} \frac { M_2^{1+\vepsilon}} {\sqrt{H_2}}, 
		\end{align*}
	and the result follows again from \eqref{6eq: >2/3}. 
		In the case when $H_1 \leqslant d_1 - c_1$ and  $H_2 > d_2 - c_2$, we use the same argument to conclude the proof. 
	\end{proof}
	
\delete{	Observe that   Lemma \ref{lem: A-process, 2.2} simplifies when $M_1 = M_2 = M$, for the condition \eqref{6eq: assumption for A, 2} becomes redundant.  This is becasue $H_1, H_2 < M^{\frac {1-\lambdaup} {\kappaup+1}} < \hskip -1.5pt \sqrt{M}$ by \eqref{6eq: H1, H2} and \eqref{6eq: assumption for A}. By induction, it is easy to deduce the analogue of Corollary \ref{cor: Aq, 2.1}.

	\begin{cor}\label{cor: Aq, 2.2}
		Let $q$ be a positive integer. Set $Q = 2^q$.
		
		{\rm(1)} We have 
		\begin{align}\label{6eq: Aq, 2}
			A^q (\kappaup, \lambdaup) = \lp \frac {\kappaup} { 2 (Q-1) \kappaup + Q}, 1 - \frac {q \kappaup -\lambdaup + 1 } {2 (Q-1) \kappaup + Q } \rp .
		\end{align} 
		
		{\rm(2)} Let $T > M^{q+1}$.    Define 
		\begin{align}
			H_q = M^{\frac {2(Q-1) ((q+1) \kappaup -\lambdaup + 1)} {2(Q-1) \kappaup + Q   } }  / T^{\frac { 2(Q-1) \kappaup } { 2 (Q-1) \kappaup + Q} }. 
		\end{align} 
		In order for $ A^q (\kappaup, \lambdaup) $ to be an exponent pair for  $(M , M  , T  , T  ,  N )$,  it suffices that $ (\kappaup, \lambdaup) $ is  an exponent pair for  any   $(M , M  ,  T h_1 /M^q  , T h_2 /M^q  , N H_q / M^q  )$   with  $1 \leqslant h_1, h_2 < H_q$. 
	\end{cor}
}

	Next, we would like to remove the condition \eqref{6eq: assumption for A} 
	in the case that the sums become separable   after applying the Weyl difference.

	\begin{lem}\label{lem: A-process, 2.3}
		Let notation be as in Lemma {\rm\ref{lem: A-process, 2.2}}. If $ N (H_1/M_1 + H_2/M_2) \Lt 1 $, then $A (\kappaup, \lambdaup)$ is an  exponent pair for   $( T_1 , T_2 , M_1, M_2 , N )$.  
	\end{lem}
	
	\begin{proof}
		Use \eqref{6eq: bound, separable} in   Lemma   \ref{lem: reduction, separable} instead of \eqref{4eq: exp pair, 2} in the proof of Lemma \ref{lem: A-process, 2.2}. 
	\end{proof}

	
For the $B$-process, the error terms in \cite[Theorem 2.24]{Kratzel}   will be too weak for our later applications, so, instead, we shall use a variant of \cite[Lemma 5.5.3]{Huxley} as follows  (see Remark \ref{rem: two vdCorput}).

\begin{lem}
	Suppose that $g (y) \in C^4 [c, d]$ and $\varww (y) \in C^1 [c, d]$ are  real functions. Let $M, T , U $ be positive parameters, with $M \geqslant d-c$, such that
	\begin{align}\label{6eq: condition, 1}
		g^{(j)} (y ) \Lt T/ M^j , \qquad \varww^{(k)} (y) \Lt U/M^k, 
	\end{align} 
for $j=2, 3, 4$, $k = 0, 1$, and
\begin{align}\label{6eq: condition, 2}
	g'' (y ) \Gt T/ M^2.
\end{align} 
Define $
	a = g'(c)$,  $  b = g' (d)$. Let $\vepsilon,  \theta -  \sqrt{T}/M -1  \in  (0, 1] $.   Then
\begin{align}\label{6eq: variation of vdC transform}
	  \begin{aligned}
	  	& \sum_{c \leqslant m \leqslant d}  \varww (m) e(g(m)) =   \sum_{a + \theta \leqslant n \leqslant b - \theta } \frac {\varww (y_n ) e(g(y_n) - n y_n  )} {\sqrt{g'' (y_n ) / i}}   \\
	  & \  +  \sum_{a- \vepsilon \leqslant n < a+\theta} + \sum_{ b- \theta < n \leqslant b+ \vepsilon } \int_{c}^{d} \varww (y) e (g(y) - n y) \nd y   + O_{\vepsilon} (U \log (b-a+2)), 
	  \end{aligned}
\end{align}
where $y_n $ is the unique value in $[c, d]$ with
\begin{align*}
	g' (y_n) = n.  
\end{align*}
\end{lem}

\begin{proof}
	The formula \eqref{6eq: variation of vdC transform} is clear from the proof of \cite[Lemma 5.5.3]{Huxley}. As for the truncated Poisson formula we use \cite[Proposition 8.7]{IK} instead of \cite[Lemma 5.4.3]{Huxley}. 
\end{proof} 
	
	\begin{lem}
	 Let $(\kappaup, \lambdaup)$ be a one-dimensional exponent pair as in Definition {\rm\ref{defn: exp pair, 1}}.   For  $g \in \boldsymbol{\mathrm{F}}_2^{\gamma}   (M_1, M_2, T_1, T_2, N  ) $ as in Definition {\rm\ref{5defn: class F(gamma)}}, we have 
	 \begin{align}\label{6eq: Poisson, 2}
	 	\begin{aligned}
	 		\sum_{(m_1, \shskip m_2) \shskip \in D} e (g  (m_1, m_2)) =    & \sum_{(n_1, \shskip n_2) \shskip \in E} \frac{e (f  (n_1, n_2))} { {f_{\snatural} (n_1, n_2)}}   + \Delta + \Delta^{\ssharp} + \Delta^{\sflat} + \Delta_1^{\oldstylenums{0}} + \Delta_2^{\oldstylenums{0}} , 
	 	\end{aligned}
	 \end{align}
	 where  $f \in \boldsymbol{\mathrm{F}}_2^{\gamma/(\gamma -1)}   (T_1/M_1, T_2/M_2, T_1, T_2, N  )$,
	 \begin{align}\label{6eq: bounds for f n}
	 	\frac {\partial^{j_1+j_2}} {\partial x_1^{j_1} \partial x_2^{j_2}} \frac 1 { f_{\snatural} (x_1, x_2) } \Lt \frac {M_1 M_2} {\sqrt{T_1 T_2}} \cdot \frac 1 {(T_1/M_1)^{j_1 } (T_2/M_2)^{ j_2}} 
	 \end{align}
	 for $j_1, j_2 = 0, 1$, 
	   $E$ is the image of $D$ under the map 
	 \begin{align*}
	 	\begin{aligned}
	 		x_1 =     {T_1 } /{M_1 } \cdot   \psi_1 '  (   {y_1} / {M_1 }  )   ,  \qquad 
	 		x_2   =   {T_2 } / {M_2 } \cdot   \psi_2 '  (   {y_2} / {M_2 }  )  ,
	 	\end{aligned} 	
	 \end{align*} 
 \begin{align}\label{6eq: Detla}
 	\Delta  = O \big( M_2    \log T_1   + \sqrt {T_1}     \log T_2   \big),   
 \end{align}
\begin{align}\label{6eq: Detla sharp}
	\Delta^{\ssharp}  = O \Big( \big( M_1/\sqrt{T_1} +1\big)  M_2^{\kappaup - \lambdaup + 1}  T_2^{\lambdaup - 1/2} + \big( M_2/\sqrt{T_2} +1\big)  M_1^{\kappaup - \lambdaup + 1}  T_1^{\lambdaup - 1/2} \Big),   
\end{align}
\begin{align}\label{6eq: Detla flat}
	\Delta^{\sflat} = O \big(\big( M_1/\sqrt{T_1} +1\big) \big( M_2/\sqrt{T_2} +1\big)\big),   
\end{align}
and
\begin{align}\label{6eq: Delta 0}
	\Delta_1^{\oldstylenums{0}} = \left\{ 
	\begin{aligned}
		\hskip -1pt & 0, & & \hskip -3pt \text{ if } N \Lt   M_1  , \\
		\hskip -1pt & O (N   {\textstyle M_2/\sqrt{T_1} }), & & \hskip -3pt \text{ if otherwise, }
	\end{aligned}\right. \quad  \Delta_2^{\oldstylenums{0}} = \left\{ 
\begin{aligned}
\hskip -1pt & 0, & & \hskip -3pt \text{ if } N \Lt   M_2  , \\
\hskip -1pt & O (N  {\textstyle   \sqrt{T_1/T_2}  })  , & & \hskip -3pt \text{ if otherwise.}
\end{aligned}\right.
\end{align}  
	\end{lem}
	
	\begin{proof}
		
		Let $\vepsilon_1, \vepsilon_2, \theta_1 - \sqrt{T_1}/M_1 -1, \theta_2 - \sqrt{T_2}/M_2- 1 \in (0, 1]$. Let $E = [a_1, b_1] \times [a_2, b_2]$, $E' = [a_1-\vepsilon_1, b_1+\vepsilon_1] \times [a_2-\vepsilon_2, b_2+\vepsilon_2]$, and $E^{\snatural} = [a_1 + \theta_1, b_1 - \theta_1] \times [a_2+\theta_2, b_2-\theta_2]$.  Partition $ E '  $ into nine rectangles with $E^{\snatural}$ at the center, and let  $E^{\ssharp}$ or $E^{\sflat}$   denote the union of  four rectangles at  the vertices or along the  sides of $E$ respectively. Moreover, suitably choose $\vepsilon_1, \vepsilon_2, \theta_1, \theta_2 $ so that the vertices of $E'$ and $E^{\snatural}$ are of distance  at least  $1/4$, say, away from integer points. 
		 
		Our idea is to apply twice the formula \eqref{6eq: variation of vdC transform} to the $m_1$- and $m_2$-sums along with rectangular regularization. For the moment, assume   
		\begin{align}\label{6eq: N<M}
			N \Lt \min  \{ M_1, M_2  \}, 
		\end{align} with small implied constant, so that no error term occurs in the process of regularization. Then the exponential sum on the left of \eqref{6eq: Poisson, 2} is transformed into the sum
		\begin{align*}
			S^{\snatural} + S^{\ssharp} + S^{\sflat} + \Delta ,
		\end{align*}
	where  $S^{\snatural}$, $ S^{\ssharp}$, and $ S^{\sflat}$ are sums over $(n_1, n_2)$ in  $E^{\snatural}$, $ E^{\ssharp}$, and $ E^{\sflat}$, respectively, and $\Delta$ is as in \eqref{6eq: Detla}.   Remarks on  the second application of   \eqref{6eq: variation of vdC transform} are in order. First, we need to change the order of summations, so the first regularization is very important. Second, we need to use Lemma \ref{lem: Bruno, 1} and \ref{lem: solution in 1-dimension}  to verify the conditions \eqref{6eq: condition, 1} and \eqref{6eq: condition, 2}.  
	
	Now we describe and analyze $S^{\snatural}$, $ S^{\ssharp}$, and $ S^{\sflat}$ in  more explicit terms. 
	
	Firstly,
		\begin{align}\label{6eq: S n}
		S^{\snatural} =	\sum_{(n_1, n_2) \in E^{\snatural}} \frac{e (f  (n_1, n_2))} { {f_{\snatural} (n_1, n_2)}}  ,
		\end{align}
	where  $f   $ and $f_{\snatural}$ are defined in the same manner as in \eqref{6eq: f =, simple}--\eqref{6eq: psi =, simple}; it suffices to know that  $f \in \boldsymbol{\mathrm{F}}_2^{\gamma/(\gamma -1)}   (T_1/M_1, T_2/M_2, T_1, T_2, N  )$ and that  $f_{\snatural}$ has bounds as in \eqref{6eq: bounds for f n} by the work in \S \ref{sec: analytic lemmas}. Note that if the summation in \eqref{6eq: S n} is extended from $E^{\snatural}$ onto $E$, then we have an extra error of the form $\Delta^{\ssharp}$ as in \eqref{6eq: Detla sharp} by applying \eqref{4eq: exp pair, 1} to either the $m_1$- or the $m_2$-sum (this step is superfluous, as in practice one may  apply \eqref{4eq: exp pair, 2} to the sum $S^{\snatural}$ directly.).
	
	Secondly, $S^{\ssharp}$ splits into four similar sums, one of which is of the form
	\begin{align}
		 S^{\ssharp}_{11} =   \sum_{a_1' \leqslant n_1 < a_1^{\snatural} } \sum_{a_2^{\snatural} \leqslant n_2 \leqslant b_2^{\snatural} } \int_{c_1}^{d_1} \frac{e (g_1 (y_1)  - n_1 y_1  - f_2 (n_2) + h_1 (y_1;   n_2) )} { {f_{\snatural 1} (y_1;   n_2)}} \nd y_1,
	\end{align}
where $a_1' = a_1- \vepsilon_1$, $a_1^{\snatural} = a_1 + \theta_1$, $a_2^{\snatural} = a_2 + \theta_2$, $b_2^{\snatural} = b_2 - \theta_2$, $g_1 \in \boldsymbol{\mathrm{F}}_1^{\gamma} (  M_1, T_1 )$, $ f_2 \in \boldsymbol{\mathrm{F}}_1^{\gamma/(\gamma -1)} ( T_1/M_1, T_1 ) $,  $$h_1 (y_1;   n_2) = T_1 \delta_1 (  y_1/M_1; M_2n_2/T_2), $$  with   $\partial_1^{j_1} \partial_2^{j_2} \delta_1  (y_1 ; x_2) \Lt_{j_1, j_2} N/T_1 $ (by Lemma \ref{lem: solution in 1-dimension}), and
\begin{align*}
	 \frac {\partial^{j_1 + j_2 }} {\partial y_1^{j_1 } \partial x_2^{j_2}} \frac 1 { f_{\snatural 1} (y_1;   x_2) } \Lt \frac { M_2} {\sqrt{ T_2}} \cdot \frac 1 {M_1^{j_1 } (T_2/M_2)^{ j_2}} 
\end{align*}
for $j_1, j_2 = 0, 1$.  
Then $S^{\ssharp}_{11}$ is bounded by the first  term of $ \Delta^{\ssharp} $ in \eqref{6eq: Detla sharp} on exploiting the one-dimensional second derivative for the $y_1$-integral and the bound \eqref{4eq: exp pair, 1} for the $n_2$-sum. A cautious reader may find a subtle issue with the `mixing' error phase $ h_1 (y_1;   n_2) $. To address this, we use the simple arguments in the proofs of \cite[Lemma 5.1.2, 5.1.3]{Huxley}: divide and partially integrate the $y_1$-integral, estimate  the resulting $n_2$-sums by \eqref{4eq: exp pair, 1} and finally the   $y_1$-integrals trivially. 

Thirdly, 
\begin{align}
	S^{\sflat} = \sum_{(n_1, n_2) \in E^{\sflat}} \int \hskip -4pt \int_D e (g  (y_1, y_2) - n_1 y_1 - n_2 y_2) \nd y_1 \nd y_2. 
\end{align}	
This yields  $ \Delta^{\sflat} $ in \eqref{6eq: Detla flat} by the two-dimensional second derivative test (see for example \cite[Lemma 4]{Srinivasan-Lattice-2}). Note that the Hessian matrix here is `almost diagonal'.

Finally, with the aid of the second derivative tests, one may verify that the rounding errors  arising from the rectangular regularizations are $\Delta^{\oldstylenums{0}}_1$ and $\Delta^{\oldstylenums{0}}_2$ as in \eqref{6eq: Delta 0} in case that \eqref{6eq: N<M} is not true. 
	\end{proof}


	

	\begin{cor}[$B$-process]\label{lem: B-process, 2}  
If $(\kappaup, \lambdaup)$ is an exponent pair for $( T_1 , T_2 , T_1/ M_1, T_2/ M_2, \allowbreak N   )$, then   $	B (\kappaup, \lambdaup)  $ is an exponent pair for $( T_1 , T_2 ,   M_1,   M_2, N  )$ if   the following   conditions hold{\rm:}  
\begin{align} 
	\label{6eq: assume from error, 1}     M_2  < (M_1 M_2)^{\kappaup - \lambdaup + 1} (T_1 T_2)^{\lambdaup - 1/2} ,
\end{align} 
	\begin{align} 
		\label{6eq: assume from error, 2}       \sqrt{T_1}   <   (M_1 M_2)^{\kappaup - \lambdaup + 1} (T_1 T_2)^{\lambdaup - 1/2},  
	\end{align}  
\begin{align}
	\label{6eq: assume from error, N}    N    <      \min   \big\{ M_1/T_1^{\vepsilon}, M_2/T_2^{\vepsilon} \big\} ,  
\end{align}  
	and   the third condition {\rm\eqref{6eq: assume from error, 1}} {\rm(}when it fails{\rm)} may be replaced by
	\begin{align}\label{6eq: assume N < M, 1}
		N  T_1   <  (M_1 M_2)^{\kappaup - \lambdaup +1} (T_1 T_2)^{\lambdaup} ,
	\end{align}   
\begin{align}\label{6eq: assume N < M, 2}
	N   M_2 \sqrt{T_2}   <  (M_1 M_2)^{\kappaup - \lambdaup +1} (T_1 T_2)^{\lambdaup} .
\end{align}  
	\end{cor}
	
	 \begin{proof}
	 	Apply \eqref{4eq: exp pair, 2} to the sum on the right of \eqref{6eq: Poisson, 2}. The error terms $\Delta^{\ssharp}$ and $\Delta^{\sflat}$ in \eqref{6eq: Detla sharp} and \eqref{6eq: Detla flat} are satisfactory as $M /\sqrt{T }  < M ^{\kappaup - \lambdaup + 1}  T ^{\lambdaup - 1/2}$ for $M < T$. The conditions  \eqref{6eq: assume from error, 1}--\eqref{6eq: assume N < M, 2}  correspond to $\Delta $, $\Delta^{\oldstylenums{0}}_1$, and $\Delta^{\oldstylenums{0}}_2$ in \eqref{6eq: Detla} and \eqref{6eq: Delta 0}. 
	 \end{proof}
 
\begin{rem}
We remark that {\rm\eqref{6eq: assume N < M, 1}} and {\rm\eqref{6eq: assume N < M, 2}} are much weaker than \eqref{6eq: assume from error, N}, but in practice the latter is easier to verify and it implies the `almost separable' condition {\rm\eqref{6eq: almost separable}}. 
\end{rem}

	\subsection{Process $A  B A B A B A$} \label{sec: ABABABA}

	Let $T = M^{\valpha}$ and $N = M^{\valpha -1}$ with $\valpha > 5/2$. We shall start with $ (M,T,M,T, N) $ and use the exponent pair $BA    ( {13} / {84},  {55} / {84} )  = (55/194, 55/97)$ at the end. Table \ref{table: act on (M,T)} and \ref{table: exp pairs} exhibit the intermediate parameters $ (M_i, T_i) $ ($i=1, 2$), the variations of $N$, and the corresponding exponent pairs $(\kappaup, \lambdaup)$ in the process of applying $A BABABA$, where, in view of Lemma \ref{lem: A-process, 2.2},
	\begin{equation}\label{6eq: bounds for hi}
		\begin{aligned}
			&1 \leqslant h_i < H, \quad 1 \leqslant h_i'  <  H' h_i^{ \frac { 401} {760} } , \quad 
			1 \leqslant h_i''   < H'' {h_i'}^{ \frac {1} {2}  } / h_i^{\frac {207} {304} },   
		\end{aligned} 
	\end{equation} 
	with
	\begin{align}
		\label{6eq: H =}
		H = M^{ \frac {1326}{1879} - \frac {359} {1879} \valpha } , \hskip 9pt H'= M^{ \frac { 401} {760} \valpha - \frac {477} {380}} , \hskip 9pt H'' = M^{\frac {283} {152} - \frac {207} {304} \valpha }.   
	\end{align}

{\renewcommand{\arraystretch}{1.5}
	\begin{table} 
		\begin{center}
			\caption{ } \label{table: act on (M,T)}
			\begin{tabular}{ c|c | c }  
				\hline
				Process   & $(M_i, T_i )$ & Variation of $N $    \\
				\hline
				& $(M,T)$ &   $1$  \\  
				$A$ & $  \lp M,   {Th_i} / {M} \rp$ & $   (h_1+h_2) / M $  \\ 
				$AB$ & $  \lp {Th_i} / {M^{2}},   {Th_i} / {M}  \rp$ &   $1$  \\ 
				$ABA$ & $  \lp {Th_i} / {M^{2}},   M h_i'  \rp$ &   $  (h_1' /   h_1 + h_2'/h_2) M^{2}/T  $   \\ 
				$ABAB$ & $  \lp {M^{3} h_i'} / {T h_i},   M h_i'  \rp$ &   $1$  \\ 
				$ABABA$ & $  \lp {M^{3} h_i'} / {T h_i},   T h_i h_i''/M^{2}  \rp$  &  $ (h_1  h_1'' /   h_1' + h_2  h_2'' /h_2' ) T/M^{3}$  \\
				$ABABAB$ & $  \lp T^2 h_i^2 h_i''    /M^{5} h_i'   ,   T h_i h_i''/M^{2}  \rp$  &   $1$ \\
				$ABABABA$ & &   \\
				\hline 
			\end{tabular}
		\end{center}
	\end{table} 
}

{\renewcommand{\arraystretch}{1.5}
	\begin{table}
		\begin{center}
			\caption{ } \label{table: exp pairs}
			\begin{tabular}{ c|c  }  
				\hline
				Process   & $(\kappaup, \lambdaup )$  \\
				\hline
				$ABABABA$ & {\Large $  \lp \frac {359} {3758}, \frac{2791}{3758} \rp$}   \\ 
				$BABABA$ &  {\Large $  \lp \frac {359} {1520}, \frac{3}{5} \rp$} \\ 
				$ABABA$ & {\Large $  \lp \frac{1}{10}, \frac {1119} {1520}  \rp$} \\  
				$BABA$ &  {\Large $  \lp \frac{1}{4}, \frac {359} {608}  \rp$}   \\ 
				$ABA$ & {\Large $  \lp \frac {55} {608}, \frac{3}{4} \rp$}    \\ 
				$  BA$ &  {\Large $  \lp \frac {55} {249}, \frac{152}{249} \rp$} \\ 
				$A$ & {\Large $  \lp \frac {55} {498}, \frac{359}{498} \rp$}  \\ 
				& {\Large $  \lp \frac {55} {194}, \frac{55}{97} \rp$}  \\
				\hline 
			\end{tabular}
		\end{center}
	\end{table} 
}

	Firstly, in order to apply Lemma  \ref{lem: A-process, 2.2} to the middle two $A$-processes, we need
	\begin{align*}
		{Th_i} / {M} > \big( {Th_i} / {M^{2}} \big)^2, \qquad M h_i' > \big( {M^{3}h_i'} / {T h_i}\big)^2, 
	\end{align*}
	as in \eqref{6eq: assumption for A}, or equivalently, 
	\begin{align*}
		h_i < M^{3}/T, \qquad h_i' /h_i^2 < T^2 / M^{5}.  
	\end{align*}
	Since $h_i < H$ and $ h_i' /h_i^2 < H' $ by \eqref{6eq: bounds for hi}, it suffices that
	\begin{align}\label{6eq: interval 1}
		 \frac{2846}{1119}  \leqslant	\valpha \leqslant   \frac {4311} {1520}   , 
	\end{align}
with ${2846}/{1119} = 2.54334...$ and $  {4311} / {1520} =2.83618...$.
	
	Secondly,  to apply Lemma  \ref{lem: A-process, 2.3} at the last step, in view of Table  \ref{table: act on (M,T)}, we need 
	\begin{align}\label{6eq: last N}
		    (h_1+h_2) \bigg( \frac{h_1'}    {h_1} + \frac{h_2'}    {h_2} \bigg)   \bigg( \frac{h_1  h_1''}    {h_1'} + \frac{h_2  h_2''}    {h_2'} \bigg)    \Bigg( \frac {{h_1'}^{\hskip -1pt \frac {55} {249} }}   {h_1^{\frac {55}{83}} {h_1''}^{\hskip -1pt \frac {110} {249} }} + \frac {{h_2'}^{\hskip -1pt \frac {55} {249} }}   {h_2^{\frac {55}{83}} {h_2''}^{\hskip -1pt \frac {110} {249} }} \Bigg) M^{\frac {28}{83} \valpha - \frac {362} {249}} \Lt 1 . 
	\end{align}
	By \eqref{6eq: bounds for hi}, we have
	\begin{align*}
	h_i < H, \quad	h_i'/h_i < H', \quad h_i  h_i'' /   h_i' < H'' H^{ \frac {97} {304}}, \quad {{h_i'}^{\hskip -1pt \frac {55} {249} }} /  {h_i^{\frac {55}{83}} {h_i''}^{\hskip -1pt \frac {110} {249} }} <   {H'}^{\hskip -1pt \frac {55} {249} },  
	\end{align*} 
	so the product on the left of \eqref{6eq: last N} is at most
	\begin{align}\label{6eq: product H}
		\Lt  H^{\frac {401}{304} } {H'}^{\hskip -1pt \frac {304}{249}}  H''  M^{\frac {28}{83} \valpha - \frac {362} {249}}  .  
	\end{align}
	Actually, this can be attained on choosing $h_1 = 1$, $h_2 = H$, $h_1' = H'$, $h_2' = 1$, $h_1''=1$, and $h_2'' = H'' / H^{\frac {207}{304}}$. 
	Numerical calculations by    \eqref{6eq: H =}  confirm that \eqref{6eq: product H} does not exceed the unity if $\valpha$ is in the range \eqref{6eq: interval 1}. 
	
	Thirdly, we  verify the `almost separable' condition  \eqref{6eq: almost separable} and the conditions \eqref{6eq: assume from error, 1}, \eqref{6eq: assume from error, 2}, \eqref{6eq: assume from error, N}, or \eqref{6eq: assume N < M, 1},  \eqref{6eq: assume N < M, 2} for the three $B$-processes. For the last   $B$-process, we  verify \eqref{6eq: assume N < M, 1} and \eqref{6eq: assume N < M, 2} in place of \eqref{6eq: assume from error, N} as it may fail (for $\valpha \geqslant 2.67653...$). Except for the last $B$- (or $A$-) process, we do not need to verify  \eqref{6eq: almost separable} as it is clearly implied by \eqref{6eq: assume from error, N}.   To this end, we use \eqref{6eq: bounds for hi} to reduce these conditions to  
\begin{align}\label{6eq: verify N}
H < M^{3-\valpha-\vepsilon}, \qquad H H' < M^{ \valpha-2-\vepsilon}, \qquad 	H^{ \frac {401}{304}} H' H'' < M^{1-\vepsilon}, \qquad 
\end{align}	  
\begin{align}\label{6eq:  verify N, 2}
	H^{  \frac {1093}{608} }  {H'}{}^{\hskip -1pt \frac 3 4}	{H''}{}^{ \hskip -1pt \frac {637}{498}} < M^{\frac {1109}{249} - \frac {111}{83} \valpha }, \qquad H^{  \frac {173}{152} }  {H'}{}^{\hskip -1pt \frac {401}{249} }    {H''}{}^{\hskip -1pt \frac {194}{249}}    
	<   M^{\frac {27}{166} \valpha + \frac {113} {249} }, 
\end{align} 
\begin{align}\label{6eq: verify error, 1}
	1	< M^{\frac {359}{760} \valpha - \frac {283}{380}}, \qquad	H^{\frac  {401}{608} } < 	M^{  \frac {173}{152} - \frac {97}{304}  \valpha } , \qquad H^{\frac {14}{83}} {H'}{}^{\hskip -1pt \frac {152}{249}} < M^{  \frac {55}{83} \valpha  - \frac {385} {249}   },
\end{align}
\begin{align}\label{6eq: verify error, 2}
	H^{  \frac {401}{1520}} < M^{\frac {93}{380} -\frac {21}{760} \valpha   }, \qquad {H'}{}^{\hskip -1pt \frac 1 4} < 	M^{\frac {207}{304}  \valpha - \frac {207}{152}} , \qquad H^{\frac {291}{608}} {H''}{}^{ \hskip -1pt \frac {139}{498}} < M^{ \frac {611} {249} - \frac {139}{166} \valpha   }.
\end{align}
More explicitly, \eqref{6eq: almost separable} and \eqref{6eq: assume from error, N} are reduced to \eqref{6eq: verify N}, \eqref{6eq: assume N < M, 1} and \eqref{6eq: assume N < M, 2} to \eqref{6eq:  verify N, 2}, \eqref{6eq: assume from error, 1}   to \eqref{6eq: verify error, 1}, and \eqref{6eq: assume from error, 2}   to \eqref{6eq: verify error, 2}, respectively. It can be checked directly that these are valid for $\valpha $ in the range \eqref{6eq: interval 1}.

	Finally, since  $ \lp   {13} / {84},  {55} / {84} \rp $ has exceptional set  $  \{    1, 3/2, 2, 5/2, 3, 7/2, 4  \}$, it is easy to determine when the starting phase exponent $\gamma $ is admissible. 

	\begin{thm} \label{thm: vdC}  For  $g \in \boldsymbol{\mathrm{F}}_2^{\gamma } ( M , M , M^{\valpha}, M^{\valpha}, M^{\valpha-1} )  $ we have
		\begin{align}\label{6eq: final estimate, q=1}
			S^2_g (M, M) \Lt_{  \gamma ,  \vepsilon}  M^{\frac {2432} {1879} + \frac {359} {1879} \valpha + \vepsilon}
		\end{align}
		if $\valpha \in [{2846}/{1119}, {4311}/ {1520}  ]$ and  $\gamma    \notin \big\{ 1, 2, 3 \big\} \cup (2 +  \{1/2,  3/5, 8/13, 5/8, 12/19, \allowbreak 7/11, \allowbreak 16/25,  9/14, 2/3 \})$. 
	\end{thm}
	
	The $\beta$-barrier of $ABABABABA \lp   {13} / {84},  {55} / {84} \rp $ is at $ 995/608 =1.63651... $, considerably improving $ 219/139 = 1.57554... $   in \S \ref{sec: simple process ABABA}. 
	
	\subsection{Remarks on the second van der Corput method}\label{sec: remarks, 2}
	
	Numerical calculations suggest that  $A ( {13} / {84},  {55} / {84} )$ has the optimal $\beta$-barrier $ 1.64545... $, and  $ABABA ( {13} / {84},  {55} / {84} )$ has the second best $\beta$-barrier $ 1.63816... $ (this could probably be confirmed by the algorithm in \cite[\S 5]{GK-vdCorput}). However, neither of these can be achieved by the method, because the $N$-parameter remains above the unity after $A$ or $ABABA$ for $\valpha > 5/2$. It might be of interest to note that the  $\beta$-barrier of $(AB)^q A ( {13} / {84},  {55} / {84} )$ decreases for $q$ even and increases for $q$ odd to the same limit as $q \ra \infty$. 
	
	For $\valpha $ large, in principle, one might expect $A^q BABABA$ ($q= 2, 3, ...$) to yield non-trivial results as in \S \ref{sec: simple process ABABA}.  However, this is not the case, because  the condition to bring the last $N$-parameter below the unity becomes too strong.    Alternatively, if we choose   $( {13} / {84},  {55} / {84} ) $ instead of $  (55/194, 55/97)$ at the end, the method would work, but only for $\valpha \in [2.54605..., 2.84046...]$ if $q=1$ and for $\valpha \in [3.70128..., 3.76069...]$ if $q=2$. Moreover, this yields a non-trivial bound for the exceptional cases when  $\gamma   \in 2 +  \{  8/13,   12/19, \allowbreak 7/11, \allowbreak 16/25,  9/14  \}$.


	\section{Proof of theorems}

	\subsection{Proof of Theorem \ref{main-theorem 1} and \ref{main-theorem 2}}\label{sec: proof by vdC}

	For either the logarithm case or the   generic monomial case for $\beta \neq 1+1/q$ {\rm(}$q=2, 3, ...${\rm)},   we have developed    in \S \ref{sec: van der Corput} the van der Corput methods of exponent pairs for the type of double sums like $S_{\psi}^2 ( N, T )$. More precisely,  on applying Theorem \ref{thm: simple vdC} and \ref{thm: vdC}, we obtain non-trivial estimates of the form
	\begin{align}\label{2eq: bound for S(N, T)}
		S_{\psi}^2 ( N, T ) \Lt (T/N)^{2\lambdaup - 2 \kappaup} T^{2 \kappaup+\vepsilon} 
	\end{align}
	for  certain exponent pairs $ (\kappaup, \lambdaup) \in [0, 1/2] \times [1/2, 1]$ depending on the value of $\valpha = \log T/ \log (T/N)$.  As \eqref{2eq: bound for S(N, T)} is non-trivial, it is necessary that
	\begin{align}\label{2eq: condition on exp pair}
		N^{  \kappaup -\lambdaup + 1 } < T^{1 - \lambdaup -\vepsilon}. 
	\end{align} 
	Substituting \eqref{2eq: bound for S(N, T)} into \eqref{2eq: S = S/T},   we obtain
	\begin{align*}
		S (\vv; n, p_1, p_2) \Lt  \frac{T^{2\lambdaup   - 1 +\vepsilon}} {N^{2\lambdaup -2 \kappaup}} + \frac 1 {N^2} \Lt \frac{T^{2\lambdaup   - 1 +\vepsilon}} {N^{2\lambdaup -2 \kappaup}} ,
	\end{align*}
	where the second inequality is clear from $T > N$, hence by  \eqref{2eq: S(p1, p2) <}
	\begin{align*}
		\left|S (n, p_1, p_2)\right| \Lt  \frac {P T^{2\lambdaup   - 1+\vepsilon} } {{N^{2\lambdaup -2 \kappaup}} \sqrt{X n}}       ,
	\end{align*}
	and, in view of \eqref{S(off)}, we have the estimate
	\begin{align}\label{2eq: S off}
		\begin{aligned}
			{S}_{\mathrm{off}}^2 & \Lt	\frac{N^{2 \kappaup - 2\lambdaup + 3 } T^{2\lambdaup   - 1+\vepsilon} \sqrt{X}   }{ P^{\star 2}  P     K  }      \mathop{\sum \sum}_{p_1,\shskip p_2 \shskip \sim P}  {\sum_{ 1\leqslant |n| \Lt N/K }}  \frac 1 {\sqrt{n}} \\
			& \Lt \frac{N^{2 \kappaup - 2\lambdaup + 3 } T^{2\lambdaup   - 1+\vepsilon} \sqrt{X}   }{    P     K  } \sqrt{\frac {N} {K} }  \\
			& = \frac{N^{2 \kappaup - 2\lambdaup + 3 } T^{2\lambdaup   - 1+\vepsilon}     }{    \sqrt{K}   }. 
		\end{aligned}
	\end{align}
	We deduce from \eqref{4eq: S(N)=S(N,X,P)}, \eqref{5eq: S<Sdiag+Soff}, \eqref{2eq: S diag}, and \eqref{2eq: S off} that
	\begin{align*}
		S_{f} (N) \Lt \lp \sqrt{KN} + \frac {T^{\lambdaup-1/2} N^{  \kappaup - \lambdaup + 3/2 }  } {  K^{1/4} }   \rp N^{\vepsilon} + \sqrt{T} N^{\vepsilon} + \frac {N} {K} + \frac {N \sqrt{K} } { P} .
	\end{align*}
	Note that $N/K < \sqrt{KN}$ and that the choice $P = \sqrt{N}$ satisfies \eqref{2eq: N<PK} because of  \eqref{2eq: sqrt N < K}.  
	Now we choose $K = N^{\frac 4 3 (\kappaup-\lambdaup+1)} T^{\frac 4 3 \lambdaup -\frac 2 3}$.   It follows from \eqref{2eq: condition on exp pair} that $K < T^{\frac 2 3 -\vepsilon}$ and hence   \eqref{5eq: K<T} is  satisfied. 
	Since  $ K > N^{\frac 4 3 \kappaup + \frac 2 3} $ by $T > N$,  \eqref{2eq: sqrt N < K} is also satisfied, and moreover, $ \sqrt{KN} > {\sqrt{T}}$ for $N^{\frac 5 3} > T$ (see \eqref{2eq: N<T}). Therefore we conclude with
	\begin{align}\label{2eq: bound for S(N)}
		S_f (N) \Lt N^{\frac 7 6 - \frac 2 3 (\lambdaup - \kappaup)} T^{\frac 2 3 \lambdaup -\frac 1 3+\vepsilon} ,
	\end{align}
	and Theorem \ref{main-theorem 1} and \ref{main-theorem 2} follow from a translation with  $\beta = \valpha / (\valpha -1)$ and $\gamma  = 0$ or $\valpha$.

	\subsection{Proof of Theorem \ref{main-theorem: Vinogradov}} 
	By applying   Vinogradov's method to the $m_1$-sum and trivial estimation to the $m_2$-sum in the double sum $ S_{\psi}^2 ( N, T ) $ in \eqref{2eq: S(p1,p2)}, we infer that 
	\begin{align}\label{7eq: Vinogradov for S(N, T)}
		S_{\psi}^2 ( N, T ) \Lt_{\psi} (T/N)^{2 - c / \gamma^2}  
	\end{align}
	for $\gamma \geqslant 4$, where   $c > 0$ is an absolute constant.   Vinogradov's method is used here in the form of \cite[Theorem 8.25]{IK}, while their constant $c = 1/2^{18}$ may be improved as the main conjecture in Vinogradov's mean value theorem is now proven in the work of Bourgain, Demeter, Guth \cite{Bourgain-Vin} and Wooley \cite{Wooley-Vin}. Theorem \ref{main-theorem: Vinogradov} follows from the same   arguments in \S \ref{sec: proof by vdC}.

	\subsection{Proof of Theorem \ref{main-theorem: Weyl}}\label{sec: Weyl}

	We have the following result  in \cite[Theorem 8.4]{IK}  by the Weyl method.
	
	\begin{lem}\label{lem: Weyl}
		Let $ k \geqslant 2$.  Suppose that $g \in C^{\infty} [ M, 2 M]  $  satisfies 
		\begin{align*}
			y^k \big| g^{(k)} (y) \big| \asymp_{k} F .
		\end{align*}
		Then for $ [c, d] \subset [M, 2 M] $ we have
		\begin{align*}
			S^1_g (M) = \sum_{c \leqslant m \leqslant d} e (g (m )) \Lt_{k} \big(F/M^k + 1/F \big)^{1/ k 2^{k-2}}  M \log 3 M . 
		\end{align*}
	\end{lem}
	
	In view of Proposition \ref{prop: non-generic}, we choose $ k = q+2$, and $F = { K^2} / {  |\vw|}$ or $ { K^4} / T |\vw|^2$ according as $q$ is odd or even. By \eqref{2eq: x >K2/N},  
	\begin{align}\label{9eq: range of F}
		K \Lt  K^2 /|\vw| \Lt N,
	\end{align}
	and hence the range of $F$ is determined. 
If we apply Lemma \ref{lem: Weyl}   to the $m_1$-sum and trivial estimation to the $m_2$-sum in the double sum $ S_{\psi}^2 ( N, T ) $ in \eqref{2eq: S(p1,p2)}, then 
	\begin{align}\label{9eq: bound for S (n, T)}
		S_{\psi}^2 ( N, T ) \Lt_{a, \shskip q}   F ^{1/ (q+2) Q} (T/N)^{2 - 1/Q +\vepsilon} + (T/N)^{2+\vepsilon} / F^{1/(q+2) Q }.
	\end{align}

	Consider first the case when $q$ is odd. 
	For convenience, we make the assumption
	\begin{align}\label{9eq: assumption on F}
		K >  N^{1/2 + 1/q},
	\end{align}
	slightly stronger than  \eqref{2eq: sqrt N < K},  so that  $F^2 \Gt (T/N)^{q+2} $ by \eqref{9eq: range of F}, and hence the first term in \eqref{9eq: bound for S (n, T)} dominates. By  \eqref{S(off)}, \eqref{2eq: S(p1, p2) <}, and \eqref{2eq: S = S/T}, along with $X = P^2 K^2/N$, we have 
	\begin{align}\label{2eq: S off, q odd} 
		{S}_{\mathrm{off}}^2 & \Lt	 {N^{  1/Q  } T^{1 - 1/Q+\vepsilon}    }   {\sum_{ 1\leqslant |n| \Lt N/K }}  (N/n)^{1/2  + 1/ (q+2)Q}   \Lt  \frac {N^{1+1/Q} T^{1-1/Q+\vepsilon} } {K^{1/2 -1 / (q+2)Q}}.  
	\end{align}
	We deduce from \eqref{4eq: S(N)=S(N,X,P)}, \eqref{5eq: S<Sdiag+Soff}, \eqref{2eq: S diag}, and \eqref{2eq: S off, q odd} that
	\begin{align*}
		S_{a, 1+1/q} (N) \Lt \lp \sqrt{KN} + \frac {N^{1/2+1/2Q} T^{1/2-1/2Q } } {K^{1/4 -1 / (2q+4)Q}}  \rp N^{\vepsilon} + \sqrt{T} N^{\vepsilon} + \frac {N} {K} + \frac {N \sqrt{K} } { P} .
	\end{align*}
	Recall that $ T = N^{1+1/q}$. Therefore we    obtain the bound \eqref{1eq: bound for S, q odd}    on choosing $P = \sqrt{N}$ and $K = T^{\frac {2   Q -2/ (q+1)} {3     Q - 2 /(q+2)} } $ and verifying   \eqref{5eq: K<T},  \eqref{2eq: N<PK}, and \eqref{9eq: assumption on F}.   
	
	Now let $q$ be even. By arguments similar to the above, we have
	\begin{align}\label{2eq: S off, q even} 
		{S}_{\mathrm{off}}^2 & \Lt	   \frac {N^{1+1/Q} T^{1-1/Q-1/ (q+2)Q +\vepsilon} } {K^{1/2 - 2 / (q+2)Q}} + \frac {N  T^{1 + 1/ (q+2)Q +\vepsilon} } {K^{1/2 + 2 / (q+2)Q}}.  
	\end{align}
	We   make the assumption
	\begin{align}\label{9eq: assumption on F, 2}
		K < N^{  3 / 4 +   1 /q} ,
	\end{align}
	so that the second term in \eqref{2eq: S off, q even} dominates.  Consequently,
	\begin{align*}
		S_{a, 1+1/q} (N) \Lt \lp \sqrt{KN} + \frac { \sqrt{N} T^{1/2+1/(2q+4) Q } } {K^{1/4 + 1 / ( q+2)Q}}  \rp N^{\vepsilon} + \sqrt{T} N^{\vepsilon} + \frac {N} {K} + \frac {N \sqrt{K} } { P} .
	\end{align*}
	Therefore we    obtain the bound \eqref{1eq: bound for S, q even}    on choosing $P = \sqrt{N}$ and $K = T^{\frac {2  (q+2)Q+2} {3   (q+2) Q +4} } $ and verifying    \eqref{2eq: N<PK}, \eqref{2eq: sqrt N < K}, and \eqref{9eq: assumption on F, 2}.

	Finally, we remark that Theorem \ref{main-theorem: Weyl} might be improved for large $q$ by the Vinogradov method, but   \cite[Theorem 8.25]{IK} must be adapted to our needs.

\end{document}